\documentclass[11pt]{amsart}
\usepackage[utf8]{inputenc}
\usepackage{amsmath,amssymb, mathrsfs}
\usepackage{lipsum}

\usepackage{bbold}
\usepackage{centernot}
\usepackage{cite}
\usepackage[dvipsnames]{xcolor}

 \usepackage{mathrsfs}
 \usepackage{bbm}
 \usepackage{enumerate}
 
 \usepackage{cancel}
 
 \usepackage{tikz}
 \usepackage{tikzcd}

\newcommand{\R}{\mathbb{R}}
\newcommand{\N}{\mathbb{N}}
\newcommand{\Z}{\mathbb{Z}}
\newcommand{\C}{\mathbb{C}}
\newcommand{\de}{\partial}
\renewcommand{\-}{\smallsetminus}

\renewcommand{\H}{\mathbb{H}}
\newcommand{\sud}{SU(2)}

\renewcommand{\a}{\alpha}

\newcommand{\f}{\varphi}
\newcommand{\e}{\varepsilon}
\newcommand{\w}{{\omega}}
\renewcommand{\S}{\C\P^1}

\newcommand{\h}{\theta}
\newcommand{\spi}[1]{{\mathcal{T}^{\otimes #1}}}

\newcommand{\q}{\tau}
\newcommand{\sudo}[1]{{\sud^{\otimes #1}}}
\newcommand{\pr}{\cdot\infty}
\newcommand{\matrsp}{\mathscr{D}^\ell}
\newcommand{\coldsp}[1]{\mathscr{D}_{\bullet ,#1}^\ell}
\newcommand{\rowdsp}[1]{\mathscr{D}_{#1,\bullet}^\ell}
\newcommand{\cold}[1]{D_{\bullet ,#1}^\ell}
\newcommand{\rowd}[1]{D_{#1,\bullet}^\ell}
\newcommand{\xold}[1]{X_{\bullet ,#1}^\ell}
\newcommand{\xowd}[1]{X_{#1,\bullet}^\ell}
\newcommand{\aold}[1]{a_{\bullet ,#1}^\ell}
\newcommand{\aowd}[1]{a_{#1,\bullet}^\ell}
\newcommand{\fold}[1]{\phi_{\bullet ,#1}^\ell}
\newcommand{\fowd}[1]{\phi_{#1,\bullet}^\ell}
\newcommand{\LS}[1]{LS[{#1}]}
\newcommand{\RS}[1]{{RS}{[{#1}]}}

\newcommand{\LSE}[1]{LS\E[{#1}]}
\newcommand{\RSE}[1]{RS\E[{#1}]}
\newcommand{\BSE}[1]{BS\E[{#1}]}
\newcommand{\BS}[1]{BS[{#1}]}
\newcommand{\TSE}[1]{\Sigma\E[{#1}]}
\newcommand{\TS}[1]{\Sigma[{#1}]}

\newcommand{\iso}{\cong}
\newcommand{\diffeo}{\simeq}

\newcommand{\law}{\overset{\text{law}}{=}}
\newcommand{\twik}{\sim_{\textrm{2-w}}}

\usepackage{pifont}

\makeatletter
\newcommand{\tpitchfork}{%
  \raise-0.1ex\vbox{
    \baselineskip\z@skip
    \lineskip-.52ex
    \lineskiplimit\maxdimen
    \m@th
    \ialign{##\crcr\hidewidth\smash{$-$}\hidewidth\crcr$\pitchfork$\crcr}
  }%
}
\makeatother

\newcommand{\vol}{\mathrm{vol}}

\renewcommand{\P}{\mathbb{P}}
\newcommand{\E}{\mathbb{E}}

\newcommand{\mC}{\mathcal{C}}

\usepackage{hyperref}
\hypersetup{
	colorlinks=true,       
	linkcolor=blue,         
	citecolor=blue}

 \usepackage[usenames,dvipsnames]{pstricks}
 \usepackage{epsfig}
 \usepackage{pst-grad} 
 \usepackage{pst-plot} 

\newtheorem{thm}{Theorem}
\newtheorem{lemma}[thm]{Lemma}
\newtheorem{cor}[thm]{Corollary}
\newtheorem{prop}[thm]{Proposition}

\theoremstyle{definition}

\newtheorem{defi}[thm]{Definition}
\newtheorem{remark}[thm]{Remark}

\newtheorem{example}[thm]{Example}
\newcommand{\be}{\begin{equation}}
\newcommand{\ee}{\end{equation}}
\newcommand{\bega}{\begin{equation}\begin{aligned}}
\newcommand{\eega}{\end{aligned}\end{equation}}

\usepackage{mathtools} 
\mathtoolsset{showonlyrefs,showmanualtags} 

\numberwithin{equation}{section}

\title{Isotropic random spin weighted functions on $S^2$ vs isotropic random fields on $S^3$}
\author{Michele Stecconi}

\begin{document}
\begin{abstract}
We show that an isotropic random field on $\sud$ is not necessarily isotropic as a random field on $S^3$, although the two spaces can be identified. The ambiguity is due to the fact that the notion of isotropy on a group and on a sphere are different, the latter being much stronger. We show that any isotropic random field on $S^3$ is necessarily a superposition of uncorrelated random harmonic homogeneous polynomials, such that the one of degree $d$ is necessarily a superposition of uncorrelated  random spin weighted functions of every possible spin weight in the range $\{-\frac{d}{2},\dots,\frac{d}{2}\}$, each of which is isotropic in the sense of $\sud$. Moreover, for a random field of fixed degree, each spin weight appears with the same magnitude, in a sense to be specified.

In addition we will give an overview of the theory of spin weighted functions and Wigner $D$-matrices, with the purpose of gathering together many different points of view and adding ours. As a byproduct of this survey we will prove some new properties of the Wigner matrices and a formula relating the operators $\eth\overline{\eth}$ and the horizontal Laplacian of the Hopf fibration $S^3\to S^2$, in the sense of \cite{Borg}.
\end{abstract}
\maketitle
\tableofcontents


\section{Introduction}
 
In this paper we compare the theory of random spin weighted functions on the sphere $S^2$, with that of random fields on the hypersphere $S^3$. We will see the two theories in the same light, but we will clarify the distinction between the corresponding notions of isotropy.
 
A function with spin weight $s\in\frac12\Z$ on $S^2$ is a section of the spin $s$ bundle (see \cite{NP66,GM10}). In the convention\footnote{There is some ambiguity in the literature, regarding the sign of $s$. We will justify our choice in Section \ref{sec:defispin}.} of this paper this is defined to be the complex line bundle of degree $2s$ and denoted by $\spi{s}\to S^2$, see Section \ref{sec:defispin}. These objects have received a lot of attention in the last years (see \cite{malya11,bamavara,BR13}), due to their application in the statistical analysis of cosmological and astrophysical data (see \cite{bamavara}), in particular related with the Cosmic Microwave Background (see \cite{malya11}). 

A convenient way to treat such objects is the so called ``pull-back approach'' (this is the point of view adopted in \cite{BR13,GM10}), which consists in the identification of the vector space of (smooth, continuous, square integrable, etc..) sections of $\spi{s}$ with a subspace of complex valued functions on $S^3$, or on $SO(3)$ if $s\in\Z$. The reason why this is possible is that under the natural maps $S^3\diffeo\footnote{With this symbol we denote diffeomorphisms.}\sud\to SO(3)\to S^2$, the pull-back of $\spi{s}$ becomes a trivial bundle $\underline{\C}$ on $\sud$. If $s\in\Z$, then the pull-back bundle is already trivial on $SO(3)$\footnote{If $s\in\Z$, then $\spi{s}$ is a true tensor power of the tangent bundle $TS^2=\spi{1}$ and $SO(3)$ is isomorphic to the frame orthonormal bundle of $S^2$.}:
\be 
\begin{tikzcd}
\underline{\C} \arrow[r]              & S^3 \arrow[d] \\
\spi{\frac12} \text{ or } \underline{\C} \arrow[r] & SO(3) \arrow[d] \\
\spi{s} \text{, with $s\in\frac12\Z$}\arrow[r]             & S^2            
\end{tikzcd}
\ee
Under this point of view, a random spin $s$ function $\sigma_X\colon S^2\to \spi{s}$ is thought as a complex random field $X\colon S^3\to \C$ on the hypersphere $S^3\diffeo \sud\subset \C^2$, with a prescribed behavior under multiplication by a phase:
\be 
X(z\cdot e^{it})=X(z)e^{-ist}.
\ee
In this case, we say that $X$ has \emph{right spin $=-s$}. The minus sign is explained by the fact that the function $X$ represents the collection of all the coordinate expressions for the section $\sigma_X$, thus it has to be interpreted as a dual object, see Remark \ref{rem:cozione}.

As in most models, we don't want the sphere $S^2$ to have special points or directions. Consequently, the random fields that we care about are only those that reflect such \emph{isotropy}. In more rigorous terms, this means that we will study the random spin weighted functions that are invariant under the automorphisms of the bundle $\spi{s}\to S^2$ induced by orientation preserving rotations, i.e. elements of the group $SO(3)$. We will explain in Section \ref{sec:defispin} how this notion of change of variables, from the point of view of random fields on $\sud\diffeo S^3$, translates to invariance in law under the composition with left multiplication by any element.
This condition is usually called \emph{isotropy} in the context of random fields on groups (compare with \cite{libro}).
On the other hand, a random field $X\colon S^3\to \C$ on a sphere is said to be \emph{isotropic} if it is invariant in law under composition with any isometry of $S^3$, i.e. any element of $SO(4)$. This latter notion is clearly stronger than the previous, indeed the round metric on $S^3$ is in fact a bi-invariant metric on the group $\sud$ and any isometry of $S^3$ is a composition of a left and a right multiplication by two elements of the group (see \cite{Ochiai1976,HatcherAT}).

One of the purposes of this paper is to compare the two above notions of isotropy, which we will call \emph{left-invariance} and \emph{bi-invariance} (we will give the precise definition in section \ref{sec:randspf}). In addition we will consider also \emph{right-invariant} random fields, so that $X$ is bi-invariant if and only if it is both left and right invariant. The significance of such comparison is that the study of bi-invariant (isotropic for $S^3$) random fields is strictly related to the study of random waves on $S^3$. 
Given a compact Riemannian manifold $M$, we will say that a random field $X\colon M\to \C$ is a \emph{monochromatic random wave} of frequency $\lambda\in\R$ if $X$ satisfies, almost surely, the Helmholtz equation for the eigenvalue $-\lambda^2$:
\be 
\Delta_{M} X=-\lambda^2 X,
\ee
with $\Delta_{M}$ being the Laplace-Beltrami operator.\footnote{Here, we are using the term \emph{monochromatic random wave} in a broad sense, whereas in the context of Riemannian geometry (see \cite{canzani2020local, zelditch2009}), the same terminology is often used to indicate that the random field is of the form $
X=\sum_{i}a_i\phi_i,
$ 
for some family of i.i.d. complex Gaussian random variables $a_i\in \C$, and with $\phi_i$ being an orthonormal basis of the eigenspace relative to the eigenvalue $-\lambda^2$. 
}
For reasons that we will explain later (see Section \ref{sec:pre}) in this paper we will take on $S^3$ the round metric of a sphere of radius $2$, so that for each $\ell\in\frac12\N$, the eigenfunctions are all those complex valued functions whose real and imaginary parts are the restriction of real homogeneous harmonic polynomials on $\R^4$, where the ones of degree $2\ell$ are relative to the eigenvalue $-\ell(\ell+1)$, for all $\ell\in\frac12\N$\footnote{We have $\Delta_{S^3}=4\Delta_{2S^3}$.}.

It is well known that any square integrable random field $X\colon S^3\to \C$ admits a \emph{spectral representation} as a sum 
\be \label{eq:decstre}
X=\sum_{\ell,i}a^\ell_i \phi^\ell_i,
\ee
for some complex  random variables $a^\ell_i$, where $\phi^\ell_i$ is an orthonormal basis of eigenfunctions of degree $2\ell$. Then, $X$ is bi-invariant (i.e. isotropic on $S^3$) if and only if the fields $X^\ell$ are jointly bi-invariant and in this case they are automatically uncorrelated.
 From the point of view of the group $\sud$, a similar statement is known under the name of Stochastic Peter-Weyl theorem (for which we refer to \cite[Proposition 5.4]{libro} and \cite[Theorem 5.5]{libro}). It says that there is a decomposition
\be \label{eq:decsud}
X=\sum_{\ell,m,s}b^\ell_{m,s}D^\ell_{m,s},
\ee
for a suitable collection of complex random variables $b^\ell_{m,s}$, where $D^\ell_{m,s}\colon S^3\to \C$ are the coefficients (indexed as in equation \eqref{eq:hypersph}, below) of the $\ell^{th}$ Wigner matrix $D^\ell\colon S^3\to U(2\ell+1)$. 
Again, the field $X$ is left-invariant if and only if the collection of fields $X^\ell=\sum_{m,s}b^\ell_{m,s}D^\ell_{m,s}$ are jointly left-invariant and, again, in this case they are automatically uncorrelated.

A key observation is that the two decompositions above are essentially the same, due to the fact that the functions
\be \label{eq:hypersph}
\phi^\ell_{m,s}=\frac{\sqrt{2\ell+1}}{4\pi}D^\ell_{m,s}, \quad \forall \ell\in\frac12\N\text{ and } m,s \in \{-\ell,-\ell+1,\dots,\ell\},
\ee
are the hyperspherical harmonics of degree $2\ell$, thus they form an orthonormal basis of eigenfunctions on $S^3$. Although this is a well known fact (see \cite{kuwa}, for instance), we will report a simple proof for completeness, see Proposition \ref{thm:Disharmonic}. An important feature of such basis is that $\phi^\ell_{m,s}$ is a function with pure left spin $-m$ and pure right spin $-s$, see Definition \ref{def:purespin}.
\begin{remark}
With a different normalization, as $L^2$ sections of $\spi{s}$, and via the pull-back correspondence, one defines the so called \emph{spin weighted spherical harmonics} $Y^\ell_{m,s}\colon S^2\to \spi{s}$, see Remark \ref{rem:spharm}. In particular the functions $Y^\ell_{m,0}\colon S^2\to \C$ are the standard spherical harmonics\footnote{The convention on the index $m$ might differ from the usual one.}.
\end{remark}
It follows the decompositions \eqref{eq:decstre} and \eqref{eq:decsud} imply that any ``isotropic'' random field is a sum of uncorrelated ``isotropic'' random waves of frequency $\sqrt{\ell(\ell+1)}$:
\be\label{eq:coeffielle}
X=\sum_{\ell\in\frac12\N}X^\ell, \qquad X^\ell=\sum_{m,s=-\ell}^\ell a^\ell_{m,s} \phi^\ell_{m,s},
\ee
for all $\ell\in\frac12\N$. This is true for both notions of isotropy: in the sense of $S^3$ (bi-invariance) and in the sense of $\sud$ (left-invariance). 
\begin{thm}\label{thm:maindell}
The field $X$ is left, right or bi invariant if and only if the fields $X^\ell$ are jointly  left, right or bi invariant, respectively. Moreover, in this case the fields $X^\ell$ are pairwise uncorrelated.
\end{thm}
To see the true difference between these notions of invariance we have to consider a further decomposition, into the spaces of functions spanned by the coefficients of each column of the Wigner matrices.
\begin{thm}\label{thm:main1}
If $X\colon S^3\to\C$ is a square integrable random field on a probability space $\Omega$, then it can be decomposed as a sum of random fields $\xold{s}$, for all $\ell\in\frac12\N$ and $s\in\{-\ell,-\ell+1,\dots,\ell\}$, such that the series 
\be\label{eq:introxls}
X=\sum_{\ell\in\frac12\N}\sum_{s=-\ell}^\ell \xold{s},\qquad \xold{s}=\sum_{s=-\ell}^\ell a^\ell_{m,s} \phi^\ell_{m,s},
\ee
 converges almost surely in $L^2(S^3)$.  Each of the fields $\xold{s}$ is a random harmonic polynomial of degree $2\ell$, i.e. an eigenfunction of $\Delta_{2S^3}$ with eigenvalue $\ell(\ell+1)$ and, at the same time, the pull-back of a section of $\spi{s}$, i.e. a random spin weighted function on $S^2$, with spin weight $s$. 
\end{thm}
To measure the relative magnitude of the component with spin $s$, in the decomposition  we introduce a probability $\E\RS{X}$ on $\frac12\Z$, see Section \ref{sec:randspf}:
\be 
\E\RS{X}(\{s\}):=\sum_{\ell\in\frac12 \N}\E\left\{\frac{\|\xold{s}\|^2}{\|X\|^2}\right\},
\ee
where $\|\cdot\|$ is the Hilbert norm of $L^2(S^3)$.
In particular, given $s\in\frac12\Z$, such probability charges the singleton $\{s\}$ if and only if $\sum_{\ell}\xold{s}\neq 0$. We stress the fact that $\E\RS{X}$ does not depend only on the marginal probabilities of the variables $a^\ell_{m,s}$, but takes into account the higher order moments of their joint probability.

The first main result of this paper states that the random fields of type $\xold{s}$ are the true building blocks of left-invariant random fields.
\begin{thm}\label{thm:mainleft}
The field $X$ is left-invariant if and only if the fields $\xold{s}$ are
jointly left-invariant.
 Moreover, the fields $\xold{s}$ and $X^{\ell'}_{\bullet, s'}$ are uncorrelated for all $\ell\neq \ell'$, while for $\ell=\ell'$  they can be correlated,
but their correlation structure have to satisfy some strict relations, stated in Theorem \ref{thm:leacor}. However, any probability on $\{-\ell,\dots,\ell\}$ can be realized as $\E\RS{X}$ by a left-invariant random field $X=X^\ell$ of degree $2\ell$.
\end{thm} 

On the other hand, the condition of of being bi-invariant is stronger and requires an equal presence of all spin weights. 

\begin{thm}\label{thm:main2}
Any bi-invariant square integrable random field $X=X^\ell\colon S^3\to \C$ of degree $2\ell\in\N$ is a superposition of random spin weighted functions of every spin weight $s\in\{-\ell,-\ell +1,\dots,\ell\}$, such that two of them can be correlated only when they have opposite spin. Moreover, each spin weight appears with the same magnitude, meaning that $\E\RS{X}$ has to be uniform, see Corollary \ref{cor:stronguniform}.
\end{thm}
\begin{remark}
The correlation of the components with opposite spin reflects that of the real and imaginary part of $X$ and it is determined by the number
\be \label{eq:number}
\E\left\{\langle \overline{X^\ell}, X^\ell \rangle_{L^2(\sud)}\right\}\in\C,
\ee
see Theorem \ref{thm:biacor}.
For instance, if $X$ is circularly symmetric, then the number \eqref{eq:number} vanishes and thus all of the random spin weighted functions are uncorrelated. A similar proposition holds for each field $X=\xold{s}$, in the situation of Theorem \ref{thm:mainleft}, see Theorem \ref{thm:leacor}. However, the correlation is not determined by the number \eqref{eq:number}, which is always $0$ since in this case the field $X$ and its conjugate are orthogonal almost surely.
\end{remark}
In particular, a bi-invariant random field of fixed degree $X^\ell$ cannot be further decomposed into simpler bi-invariant fields, contrary to what happens in the left-invariant case. The explanation of this phenomenon is to be found in the decomposition of the space $L^2(\sud)$ into irreducible components for the action of $SO(4)$ in comparison to that of $\sud$, see Proposition \ref{prop:dirrep}. 
\begin{remark}
From this we see that an isotropic random spin function $S^2\to \spi{s}$ is \emph{never} isotropic as a random field in the sense of $S^3$, because it has a fixed spin. 
\end{remark}
We will actually prove stronger statements than the above, see Theorem \ref{thm:weakuniform} and Theorem \ref{thm:stronguniform}. In partiuclar, the latter implies that any bi-invariant monochromatic random wave $X=X^\ell$ of frequency $\sqrt{\ell(\ell+1)}$ is necessarily supported on the whole eigenspace of $\Delta_{2S^3}$ relative to the eigenvalue $\ell(\ell+1)$. More precisely, $X^\ell$ is a superposition of monochromatic random waves with pure left and right spin, such that two of them can be correlated only when both their right and left spin are opposite.
Moreover, each pair of right and left spin $(m,s)\in\{-\ell,-\ell+1\dots \ell\}^2$ appears with the same magnitude, see Theorem \ref{thm:stronguniform}.
\[ \sim \bullet\sim
\]
A second purpose of this paper is to give an overview of the theory of spin weighted functions and of the special properties of the Wigner functions, trying to gather together various different points of view. This is the content of Sections \ref{sec:pre} and \ref{sec:wigner}.

Traditionally, the Wigner functions $D^\ell_{m,s}\colon S^3\to \C$ (see section \ref{sec:wigner}), are indexed by $(\ell,m,s) \in \hat{S}^3$\footnote{If $G$ is a compact group, the notation $\hat{G}$ stands for its dual, i.e. the collection of all isomorphism classes of irreducible unitary representations of $G$.}, where 
\be\label{eq:GdualG}
\hat{S}^3:=\left\{(\ell,m,s)\in\frac12\Z^3\colon \ell\pm m,\ell \pm s \in\N\right\}.
\ee
As mentioned above, the Wigner functions are an orthogonal basis of $L^2(S^3)$ that enjoys many special properties. The most important for us is that for any fixed $\ell$ they are, at the same time, a basis of an eigenspace of $\Delta_{\sud}$ (see equation \eqref{eq:hypersph}), and the coefficients of an irreducible unitary matrix representation $D^\ell=(D^\ell_{m,s})_{m,s}$ of $\sud$, see \cite[Theorem 3.14]{libro}.
At the same time, $D^\ell_{m,s}$ is a function with pure left and right spin (see Definition \ref{def:intrispin})  and an eigenfunctions of the spin Laplacian, i.e. the operator $\eth \overline{\eth}$ constructed from the spin raising and spin lowering operators (see \cite{GM10}), thus by normalizing them in the space $L^2(S^2,\spi{s})$ one gets the so called \emph{spin weighted spherical harmonics} $Y^\ell_{m,s}$. In fact, we will show that the following formula holds when $\sud$ is endowed with the metric of the sphere of radius $2$.
\be\label{eq:decla}
\Delta_{\sud}=\eth\overline{\eth}+\frac{d^2}{d\psi^2}+\sqrt{\frac{d^2}{d\psi^2}}.
\ee

By theorem \ref{thm:main1}, the law of a square integrable random field $X\colon S^3\to\C$ is characterized by the family of random variables $a^\ell_{m,s}\in\C$, called \emph{spectral} or \emph{Fourier coefficients}. In section \ref{sec:randspf} we will study the correlation structure of these random variables in the cases of left, right and bi invariance, resulting in Theorems \ref{thm:leacor}, \ref{thm:riacor} and \ref{thm:biacor}. Theorems \ref{thm:main1}, \ref{thm:mainleft} and \ref{thm:main2} are proved as a consequence of this study. Moreover, we show the following by combining the latter theorems with the results of \cite{bamavara}.

\begin{cor}\label{cor:main3}
Assume that the coefficients $a^\ell_{m,s}$ of a square integrable random field $X\colon S^3\to \C$ are independent and centrally symmetric. Then, $X$ is bi-invariant if and only if it is complex Gaussian and the variance of $a^\ell_{m,s}$ depends only on $\ell$.
\end{cor}

The interest of Corollary \ref{cor:main3} is in the fact that it characterizes the fields $X$ that are a superposition of independent isotropic Gaussian monochromatic random waves in terms of conditions that, apparently, have nothing to do with Gaussianity. Notice that the theorem implies the independence of the real and imaginary parts of $a^\ell_{m,s}$.

Many authors already devoted their attention to the study of isotropic random fields (i.e. left-invariant in the language of the present paper) in terms of the collection of spectral coefficients, see for instance \cite{BR13, libro, bamavara,BalTra} and proved equivalent statements to Theorem  \ref{thm:leacor}. The analogous result, Theorem \ref{thm:riacor}, for right-invariance is obtained by changing the perspective and Theorem \ref{thm:biacor} for the case of bi-invariance is obtained by combining the previous two. 
We will nevertheless show how to obtain such characterization by studying the problem from the more abstract point of view of collections of random vectors $V_i\in\C^{2\ell_i+1}$ that are jointly invariant in law under the action of the collection of matrices $D^\ell$. We call such notion \emph{D-invariance} and study it in details in section \ref{sec:dinva}. This is convenient in that it allows, essentially, to study left and right-invariance at the same time, thanks to Lemma \ref{lem:ainva}. In this context, we observe the following. 
\begin{thm}\label{thm:introVDgamma}
A collection $\mathcal{V}=(V_i)_{i\in I}$ of random vectors $V_i\in \C^{2\ell_i+1}$ is (strongly) $D$-invariant if and only if
\be 
(V_i)_{i\in I}\law (D^\ell(\gamma)V_i)_{i\in I},
\ee
where $\gamma\in\sud$ is a uniformly distributed random element, independent from $\mathcal{V}$. 
\end{thm}
The above result holds also for a weaker notion of $D$-invariance, that will be specified in Section \ref{sec:dinva}.
This allows to obtain a characterization of the correlation structure of any $D$-invariant random vector, see Theorem \ref{thm:allDinvcor}. From this, we deduce Theorems \ref{thm:leacor}, \ref{thm:riacor} and \ref{thm:biacor}. In fact, thanks to Theorem \ref{thm:introVDgamma}, the problem is reduced to the computation of the correlation structure of the collection of random matrices $D^\ell(\gamma)$.
\begin{thm}\label{thm:introDcorr}
The correlation structure of the collection of random matrices $\left(D^\ell(\gamma)\right)_{\ell\in\frac12 \N}$ is described by the following identities.
\be\label{eq:introDcor}
\E\left\{D^\ell_{m,s}(\gamma)\left(\overline{D^{\ell'}_{m',s'}(\gamma)}\right)\right\}=\frac{\delta_{\ell,\ell'}\delta_{s,s'}\delta_{m,m'}}{2\ell+1};
\quad
\E\left\{D^\ell_{m,s}(\gamma)D^{\ell'}_{m',s'}(\gamma)\right\}=\frac{\delta_{\ell,\ell'}\delta_{-s,s'}\delta_{-m,m'}  (-1)^{2\ell -m+s}}{2\ell+1}.
\ee
\end{thm} 
We point out that the first set of identities \eqref{eq:introDcor} is a reformulation of Schur's orthogonality relations and these identities are valid for any irreducible unitary matrix representation $D^\ell$, while the second set derives from special properties of the Wigner functions, thus they are due also to the choice of basis in which the representation is written.

In the Section \ref{sec:defispin}, we review the theory of spin weighted functions, starting from the definition of the spin weight given by Newman and Penrose in their seminal paper \cite{NP66}. We put a special focus on the Riemannian aspects of the pull-back correspondence, showing that under the hypothesis that the $\sud\iso 2S^3$ is considered as a sphere of radius $2$, there are Riemannian coverings 
\be\label{eq:Riemanco}
\sud\to \frac{1}{|s|}S(\spi{s})=\left\{v\in\spi{s}\colon \|v\|=\frac{1}{s}\right\},
\ee
for every $s\in\frac12\Z\-\{0\}$. 
 Thanks to this, we can compare the spin Laplacian $\eth\overline{\eth}$ with the decomposition of the Laplace-Beltrami operator into a vertical and a horizontal parts described in \cite{Borg}, finding that the horizontal Laplacian relative to the Riemannian submersion $2S^3\iso\sud\to S^2$ is the operator
\be \label{eq:kuwa}
\Delta_h=\eth\overline{\eth}+\sqrt{\frac{d^2}{d\psi^2}};
\ee
Notice that $\Delta_h$ preserves the spaces of spin weighted functions, indeed it was proved by Kuwabara in \cite{kuwa} that the operator $\Delta_h=\Delta_{\sud}-\frac{d^2}{d\psi^2}$ is in fact the Bochner Laplacian acting on smooth sections of $\spi{s}$ and relative to the Chern connection.

With Proposition \ref{prop:sfera}, we will show that taking the sphere of radius $2$ is not really a choice, in that the factor $2$ is built in the structure of the Hopf fibration $S^3\to S^2$, which is a Riemannian submersion if  and only if the radius of the first sphere is the double of that of the second. This is not a novelty, but it explains why in the decomposition of the Laplacians, given in \eqref{eq:decla}, there are no factors, whereas in the literature, similar equations are written with a factor $\frac14$ in front of $\Delta_{S^3}=4\Delta_{2S^3}$, see \cite[equation (5.4)]{kuwa}.

In Section \ref{sec:wigner} we turn our attention to Wigner functions, giving a overview of their many properties.
In preparing this survey, we noticed a few results in this context, that we weren't able to find in the literature. One is Theorem \ref{thm:Dcorr} above and the other is a description of the maps 
\be\label{eq:cold}
\cold{s}\colon \sud\to S^{4\ell+1}
\ee
given by the columns of the Wigner matrices (analogous things can be said about the rows). Let $\ell\in \frac12\N$ and let us identify $\C^{2\ell+1}$ with the complex vector space generated by vectors $e_s$ indexed by $s\in\{-\ell,-\ell+1,\dots,\ell\}$. Then the map $\cold{s}$ parametrizes the orbit of the element $e_s$ of the canonical basis of $\C^{2\ell+1}$, under the unitary action $D^\ell\colon \sud \to U(2\ell+1)$. By studying these maps we deduce properties of all the orbits $O^\ell(v)=\{D^\ell(g)v\colon g\in\sud\}$ of the action. We prove the following facts.
\begin{thm}
The following things are true.
\begin{enumerate}
\item The map \eqref{eq:cold} induces an embedding $\frac{1}{|s|}S(\spi{s})\diffeo O^\ell(e_s)\subset S^{4\ell+1}$ for all $s\in\{-\ell,\dots \ell\}\-\{0\}$.
\item For $s=0$ there are two cases: if $\ell\in 2\N$, then $O^\ell(e_0)\diffeo S^1$; while $ O^\ell(e_0)\diffeo S^2$ otherwise.
\item For almost every $v\in \C^{2\ell+1}$, the orbit $O^\ell(v)$ is diffeomorphic to $\sud$ when $\ell\notin \N$ and to $SO(3)$ when $\ell\in\N$. 
\item Let $v\in\C^{2\ell+1}$ be any vector and let $n$ be the dimension of its orbit $O^\ell(v)$. If $A\subset O^\ell(v)$ is a measurable subset of positive $n$-dimensional volume, then $\text{span}(A)=\C^{2\ell+1}$.
\end{enumerate}
\end{thm}
\subsection{Aknowledgements}
This paper originated from a discussion with Antonio Lerario, Domenico Marinucci and Maurizia Rossi. The author would like to thank them for the many valuable suggestions, corrections and remarks.
\section{Spin bundles}\label{sec:pre}
\subsection{Preliminary definitions and notations}\label{sec:nota}
Given two Riemannian manifolds, we will use the symbol $\diffeo$ for diffeomorphisms and the symbol $\iso$ for isometries.

The field of quaternions $\H$ will be represented as the following space of matrices:
\be 
\H=\left\{h(\a,\beta):=\begin{pmatrix}
\a & -\overline{\beta} \\
\beta & \overline{\a}
\end{pmatrix}\Big| \a,\beta\in\C \right\}\diffeo \C^2\diffeo\{x^0+y^0\underline{i}+x^1 \underline{j}+y^1\underline{k}|x^i,y^i\in\R\}\diffeo \R^4,
\ee
so that its algebraic generators $\underline{i},\underline{j},\underline{k}$ are defined by the following identity
\be \label{eq:HR4}
x^0+y^0\underline{i}+x^1 \underline{j}+y^1\underline{k}=(x^0+y^0\underline{i})+(x^1 +y^1\underline{i})\underline{j}=\a+\beta \underline{j}.
\ee
For reasons that will be clear later (see Remark \ref{rem:8}), we will need to put the enlarged metric $\langle\cdot,\cdot\rangle_\H=4\langle\cdot,\cdot\rangle_{\C^2}$ on $\H$, i.e. the metric defined by the scalar product:
\be 
\langle h_1,h_2\rangle_\H:=2\Re \ \text{tr}\left(\overline{h_1}^T\cdot h_2\right)=\text{tr}\left(\overline{h_1}^T\cdot h_2+\overline{h_2}^T\cdot h_1\right).
\ee
With such representation, we obtain $\sud=\{g\in \C^{2\times 2}\colon \overline{g}^Tg=\mathbb{1}, \det(g)=1\}$ as the subgroup of length $2$ quaternions, i.e. the sphere of radius $2$ in $(\H,\langle\cdot,\cdot\rangle_\H)$.
\bega
\sud&=\left\{h(\a,\beta):=\begin{pmatrix}
\a & -\overline{\beta} \\
\beta & \overline{\a}
\end{pmatrix}\Big| \a,\beta\in\C, |\a|^2+|\beta|^2=1 \right\}
\\
&=\{h\in\H\colon \|h\|_{\C^2}=1\}=\{h\in\H\colon \|h\|_{\H}=2\}
\\
&\cong 2S^3.
\eega
\begin{remark}\label{rem:8}
The reason for taking the sphere of radius $2$ is that this is the only way to make the Hopf fibration a Riemannian submersion onto the standard round sphere $S^2$, see Proposition \ref{prop:sfera}.
\end{remark}
As it is well known, the Lie group $\sud$ is isomorphic to $\textrm{Spin}(3)$, the universal covering Lie group of $SO(3)$:
\be 
SO(3)=\left\{R\in \R^{3\times 3}\colon R^TR=\mathbb{1}\right\}= \left\{R=(u,v,p)\text{ positive orthonormal basis of $\R^3$}\right\}\diffeo T^1S^2\diffeo \R\P^3.
\ee
Of course, $SO(3)$ is also identified with the group of all Riemannian isometries of the round sphere $\S=\C\cup\{\infty\}\cong S^2$. 
Points in the Riemann sphere $\S$ will be denoted as
\be 
[z_0:z_1]=\frac{z_0}{z_1}=\zeta=\frac{1}{\eta}\in\S,
\ee
where $z_0,z_1\in\C^2\-\{0\}$, $\zeta,\eta\in \C\cup\{\infty\}$. In particular,
\be 
[1:0]=\frac{1}{0}=\infty \quad \text{and} \quad [0:1]=\frac{0}{1}=0.
\ee
As a Riemannian manifold, the Riemann sphere $\S$, with the Fubini-Study metric is isometric\footnote{The Fubini-Study metric on $\C\P^1$ appears in the literature with various different normalizations, which differ by a constant factor. We take the normalization such that the volume of $\S$ is $4\pi$.} to the standard round sphere $S^2=\{p\in\R^3\colon |p|=1\}$ via the bijection $\Phi:S^2\to \S$ defined by the stereographic projections from the 
north ($e_3\mapsto\infty$) pole, to the equatorial plane.
\be \label{eq:stereo}
S^2\ni p(\f,\h)=\begin{pmatrix}
x \\ y \\ t
\end{pmatrix}\footnote{We choose the letter $t$ instead than $z$ for the vertical coordinate in order to stress the fact that $t\in\R$ and not in $\C$.}=\begin{pmatrix}
\cos(\f)\cos(\h)  \\ \sin(\f)\sin(\h) \\ \cos(\h)
\end{pmatrix}\mapsto \quad\begin{aligned} \zeta&=\frac{x+iy}{1-t}=\cot\left(\frac{\h}{2}\right)e^{i\f}
\\
\eta&=\frac{x-iy}{1+t}=\tan\left(\frac{\h}{2}\right)e^{-i\f}
\end{aligned} \in \S.
\ee
Notice that 
\be 
\Phi(e_1)=1; \quad \Phi(e_2)=i;\quad \Phi(e_3)=\infty; \quad \Phi(-e_3)=0.
\ee
\begin{remark}\label{rem:polar}
As usual, the orientation of $S^2$ is defined by considering $e_1,e_2$ to be a positive frame in $T_{e_3}S^2$, i.e. taking the point of view of a \emph{polar bear} who's looking down at the north pole below him. The stereographic projection defined in \eqref{eq:stereo} defines the opposite orientation, the one for which the  same frame $\{e_1,e_2\}$ is a positive basis of $T_{-e_3}S^2$. Indeed, for points $\Phi(x,y,t)=\zeta$ near $\Phi(-e_3)= 0\in\C$ we have
\be 
\zeta= x+iy+O(|\zeta|^2).
\ee
In other words, $S^2$ is oriented from the point of view of a \emph{hamster} who's looking down while running (or not) inside a hollow sphere. The hamster and the polar bear measure the same angle between a given pair of tangent vectors, but with opposite sign.
In this paper we want $S^2$ to have the \emph{polar bear} orientation and a coherent complex structure. At the same time we choose to use the stereographic projection in the form defined in \eqref{eq:stereo}, because this is the most frequent form in the literature, see \cite{GM10}.
For this reason, we regard $\Phi$ as an antiholomorphic map. To recall this fact, sometimes we will write
\be 
\Phi\colon S^2\to\overline{\S}.
\ee
\end{remark}
\subsection{Riemannian submersions}
To have a clear view of the double covering map $\sud \to SO(3)$, let us fix an isometric action of $\sud $ on $\S$. Since $\S$ is the space of all complex lines $\ell\subset\C^2$ and $\H$ acts on $\C^2$ via $\C-$linear automorphisms, there is an obvious action of $h\in \H$ on $\S$, defined by $L\mapsto h(L)$. In terms of our chosen coordinates, the action of $h=h(\a,\beta)$ is expressed by a M\"oebius transformation:
\be 
M(h)=M(\a,\beta):\frac{z_0}{z_1}\mapsto \frac{\a z_0-\overline{\beta} z_1 }{\beta z_0+\overline{\a}z_1}.
\ee
A standard exercise (left to the reader) is to prove that such diffeomorphism is an isometry precisely when $|\a|^2+|\beta|^2=1$. Therefore, the above expression determines uniquely a homomorphism $M\colon \sud\to SO(3)$. Since its kernel is $\{\pm \mathbb{1}_2\}$ and $\sud\cong S^3$ is simply connected, the map  $M$ is the universal cover of $SO(3)$, hence it is equivalent to $\text{Spin}(3)\to SO(3)$.

The action of $SO(3)$ on the sphere is transitive, so that, choosing to view the sphere as the orbit of the point $\infty\in\S$, we get two compatible principal circle bundles over $S^2$:
\be \label{eq:diagroups}
\begin{tikzcd}
2S^1 \arrow[r, hook] \arrow[d, ":2"] & SU(2) \arrow[d, ":2"] \arrow[r, "\cdot \infty"] & \overline{\C\P^1} \arrow[d, "\cong" description] \\
S^1 \arrow[r, hook]       & SO(3) \arrow[r, "\cdot e_3"]                    & S^2                                      
\end{tikzcd}
\ee
where $\sud\to S^2$ is the map $g\mapsto M(g)\infty$, corresponding, via $\Phi$, to the Hopf fibration $S(\C^2)\to \S$:
\be\label{eq:elegant}
h(\a,\beta)\mapsto \zeta=\frac{\a}{\beta},
\ee 
while $SO(3)\to S^2$ is given by the action on the north pole $e_3=\Phi^{-1}(\infty)$, i.e. the map $R\mapsto R e_3$. 

Let the space $\R^{N^2}$ be identified with the space $\R^{N\times N}$ of square matrices of order $N$. Then it is easy to see that the Euclidean metric can be written as
\be 
g_{\R^{N^2}}\langle A,B\rangle:=\text{tr}(A^TB).
\ee
This metric is invariant under left and right multiplication by any matrix $g=g^{-T}\in SO(N)$, in that the trace is invariant under conjugation:
\be 
\text{tr}\left((g_1A g_2)^T(g_1Bg_2)\right)=\text{tr}\left(g_2^{-1}A^TBg_2\right)=\text{tr}(A^TB).
\ee
From this, it follows that its restriction to any subgroup $G\subset SO(N)$ is a bi-invariant metric. In particular, the metric induced by the inclusions $SO(3)\subset \R^9$ and $\sud\subset SO(4)\subset \R^{16}$ are bi-invariant. We also get an inclusion $\R^8=\C^{2\times 2}\subset\R^{4\times 4}$ by identifying $\C$-linear endomorphisms of $\C^{2}$ with the set of $\R$-linear endomorphisms of $\R^4$ that satisfy the Cauchy-Riemann equations. Since this inclusion has a diagonal form, the induced metric on $\R^8$ is doubled: $g_{\R^{16}}=2g_{\R^8}$. The same happens with the inclusion $\R^4\diffeo\H\subset \C^{2\times 2}$, so that
\be 
g_{\R^{16}}=2g_{\R^8}=4g_{\R^4}.
\ee
From this we see that the metric induced on $\sud$ by the inclusion in $SO(4)\subset \R^{16}$ corresponds to the round metric of a sphere of radius $2$. It turns out that such metric is the only one for which the Hopf fibration $\pr\colon\sud\to S^2$ is a Riemannian submersion.
\begin{prop}\label{prop:sfera}
Assume that $\S$ is given the round metric of radius $r$, namely $\S\cong r S^2$. There is a unique choice of bi-invariant Riemannian metrics on $\sud$ and $SO(3)$ such that all the maps in the diagram \eqref{eq:diagroups} are Riemannian submersions. Such choice is $2r S^3\cong\sud\subset (\C^2,4r^2g_{\R^{4}})$ and $SO(3)\subset (\R^{3\times 3},\frac{r^2}{2}g_{\R^9})$. In particular the lengths of the fibers are $4\pi r$ and $2\pi r$.
\end{prop}
\begin{proof}
The metric on $SO(3)$ defined by the inclusion $SO(3)\subset \R^{3\times 3}$ (with its standard metric) is bi-invariant. Choosing an orthonormal basis of $T_1SO(3)$ we get an identification of the Lie algebra $T_1SO(3)\cong \R^3$, such that the adjoint action of $SO(3)$ on $T_1SO(3)$ is isometric and, moreover, it is given by the identity: $SO(3)\to SO(\R^3)$. Any bi-invariant metric on $SO(3)$ is thus defined by a metric on $\R^3$ which is invariant by the action of $SO(3)$, but there is only one such metric, up to a constant factor. The map $SU(2)\to SO(3)$ is a Riemannian submersion if and only if it is a local isometry. This implies that the metric on $SO(3)$ defines uniquely the metric on $SU(2)$ and viceversa. Moreover, the lifted metric on $SU(2)$ is bi-invariant and the same argument implies that it is the only bi-invariant metric on $SU(2)$, up to a constant factor. Since the round metric, obtained from the inclusion $\sud\subset \R^4$ is bi-invariant, it follows that there is a unique choice, corresponding to $\sud\cong aS^3$ for some $a$ that can be determined by a computation of the volumes. Indeed, since the map $\sud\to \S$ is a Riemannian submersion, by using the coarea formula we deduce that
\be 
2\pi^2 a^3=\vol(aS^3)= (2\pi a) \cdot (4\pi r^2),
\ee 
from which we obtain $a=2r$. Here, we are also using the fact that the fibers of the Hopf fibration $\sud\cong aS^3\to\S$ are geodesic circles and thus have length $2\pi a$. Finally, we get the normalization of the metric on $SO(3)$ by observing that the fibers
of the map $\cdot e_3\colon SO(3)\to S^2$ must have length $2r\pi$, given that the leftmost vertical map in the diagram \eqref{eq:diagroups} is a Riemannian submersion. Thus the following tangent vector $v$ must have length $r$, being a generator with ``period'' $2\pi$ of the fiber over $e_3\in S^2$: 
\be v= \begin{pmatrix}
0 & -1 &0\\
1 &0 &0\\
0& 0 & 0
\end{pmatrix}\in T_{\mathbb{1}}SO(3)\subset \R^9.
\ee
\end{proof}
A similar reasoning can be applied to define the orientation of $\sud$ and $SO(3)$, by observing that by taking a point $p\in S^2$, the fiber of over $p$ of both projections can be canonically embedded as a small loop around $p$ and thus defines an orientation $S^2$.
\begin{prop}
There is a unique choice of orientations on $\sud$ and $SO(3)$ such that all the vertical maps are, locally, orientation preserving and such that the horizontal sequences define the standard \emph{polar bear} orientation on $S^2$.
\end{prop}
In the following, we will always consider $\sud$, $SO(3)$ and $S^2$ with the orientation provided by the above proposition, which we will keep referring to as the \emph{polar bear orientation}.
\subsection{Euler angles}
All the matrices $h(\a,\beta)\in\sud$ can be written in terms of Euler angles, via the following surjective parametrization:
\be 
\a=\cos\left(\frac{\h}{2}\right)e^{i\frac{\f}{2}}e^{i\frac{\psi}{2}}; \quad
\beta=\sin\left(\frac{\h}{2}\right)e^{-i\frac{\f}{2}}e^{i\frac{\psi}{2}},
\ee
with $\f\in [0,2\pi ]$, $\h\in [0,\pi]$ and  $\psi\in [-2\pi,2\pi]$. Notice that then $\f,\h$ are precisely the polar coordinates of the point $p(\f,\h)=g(\f,\h,\psi)\infty=R(\f,\h,\psi)e_3$ (see the identities \eqref{eq:stereo}).
\bega\label{eq:Eule}
g(\f,\h,\psi)&:=\begin{pmatrix}
e^{i\frac{\f}{2}} & 0 \\  0 & e^{-i\frac{\f}{2}}
\end{pmatrix}
\begin{pmatrix}
\cos\left(\frac{\h}{2}\right) & -\sin\left(\frac{\h}{2}\right) \\  \sin\left(\frac{\h}{2}\right) & \cos\left(\frac{\h}{2}\right)
\end{pmatrix}
\begin{pmatrix}
e^{i\frac{\psi}{2}} & 0 \\  0 & e^{-i\frac{\psi}{2}}
\end{pmatrix}
\\
&=: g_3(\f)g_2(\h)g_3(\psi).
\eega
\begin{remark}\label{rem:difference}
Such notation for the Euler angle in $\sud$ is different to to that of \cite[Sec 3.2]{libro}. More precisely, one is obtained from the other after the transformation
\be 
\beta\mapsto -\overline{\beta} \quad \text{ i.e. }\quad  \h\mapsto -\h.
\ee
We made this change in order to have an elegant formula \eqref{eq:elegant} for the Hopf fibration $\sud\to \S$.
Nevertheless, the convention Euler angles $\f,\h,\psi$ for $SO(3)$ is the same as in the book \cite{libro}, namely the so called $zyz$ convention. 
\end{remark}
The image of $g_3(\psi)$ and $g_2(\h)$ under the quotient $\sud\to SO(3)$ are the matrices of the standard rotations around $e_3$ and $e_2$:
\begin{prop}
\be 
g_3(\psi)\mapsto R_3(\psi)=\begin{pmatrix}
1 & 0 &0\\
0 &\cos(\psi) &-\sin(\psi)\\
0& \sin(\psi) & \cos(\psi)
\end{pmatrix}; \quad g_2(\h)\mapsto R_2(\h)=\begin{pmatrix}
\cos(\h) &0&\sin(\h)\\
0&1&0\\
 -\sin(\h) &0& \cos(\h)
\end{pmatrix}.
\ee
Therefore
\be 
g(\f,\h,\psi)\mapsto R(\f,\h,\psi):=R_3(\f)R_2(\h)R_3(\psi)
.\ee
\end{prop}
\begin{proof} It is sufficient to check that
$g_2(\frac{\pi}{2})\infty=1$ and $g_2(\h)i=i$. Moreover, $g_3(\psi)\infty=\infty$ and $g_3(\frac{\pi}{2})1=i$.
\end{proof}
To get a surjective parametrization of $SO(3)$, it is sufficient to take all matrices of the form $R(\f,\h,\psi)$, with $\f\in[0,2\pi]$, $\h\in [0.\pi]$, $\psi\in [0,2\pi]$. Indeed $R(\f,\h,\psi)=R(\f,\h,\psi+2\pi)$, while $g(\f,\h,\psi)=-g(\f,\h,\psi+2\pi)$.
\begin{remark}
The principal bundle structure in \eqref{eq:diagroups} is the one defined by the right multiplication by matrices of the form $g_3(\psi)$ and $R_3(\psi)$. From this point of view, the leftmost vertical map in the diagram corresponds to the double covering 
\be 
\R/{4\pi\Z}\to \R/{2\pi\Z}.
\ee
\end{remark}
\subsection{Definition of the Spin-weighted functions}\label{sec:defispin}
Newman and Penrose define the spin weight as follows \cite{NP66}: \emph{a quantity $u$ defined on $\mathbb S^2$ has spin weight $s$  if, whenever a tangent vector $\rho$
 at any point $x$ on the sphere transforms under coordinate change  by
$\rho'=e^{i \psi} \rho$, then the quantity at this point $x$ transforms
by  $u'=e^{is\psi} u$}.

In \cite{GM10} the authors introduce the mathematical model for spin weighted functions, viewing them as sections of complex line bundles on $S^2$. Similar approaches have been taken in \cite{BR13, malya11}.
From the statement of Newman and Penrose, it is immediately clear that a function on $S^2$ with spin weight equal to $1$ should be a section of the bundle $\spi{1}:=TS^2$ endowed with its standard complex structure. Moreover, it is also clear that a spin $s$ function is a section of $\spi{s}=\spi{1}\otimes_\C\dots \otimes_\C \spi{1}$ ($s$ times), because the transition functions of the latter bundle are, by definition, the $s^{th}$ power of those of $\spi{1}$. This is true for all $s\ge 0$. When $s\le 0$ we can argue in the same way, after noticing that the transition functions for $\spi{-1}$ are the inverse of those of $\spi{1}$, thus $\spi{-1}=(TS^2)^*$ is the dual (i.e. the inverse in the group of all line bundles) of $TS^2$. 
\begin{defi}\label{def:intrispin}
Let $s\ge 0$, we define the bundles:
\bega 
\spi{s}&:=(TS^2)^{\otimes s}=TS^2\otimes_\C\dots \otimes_\C TS^2, \quad \text{$s$ times;}
\\
\spi{-s}&:=(TS^2)^{\otimes -s}=\left((TS^2)^*\right)^{\otimes s}=(TS^2)^*\otimes_\C\dots \otimes_\C (TS^2)^*, \quad \text{$s$ times}.
\eega
\end{defi}
\begin{remark}
Complex line bundles on the sphere are classified by their Chern class $c_1\in H^2(S^2,\Z)\cong \Z$, or equivalently by their Euler characteristic $ \chi $, when thought as real oriented rank $2$ vector bundles (they are related by $c_1\frown [S^2]=\chi$). 
In fact, it is well known that in general the set of isomorphism classes of line bundles form an abelian group in which the opposite element of $L$ is $L^*$. In this case the group is isomorphic to $ \Z$ and the isomorphism is given exactly by the Euler characteristic.

For this reason, an equivalent way to state Definition \ref{def:intrispin} (up to isomorphism) is to say that $\spi{s}$ is \emph{the} complex line bundle on $S^2$ with 
\be 
\chi\left(\spi{s}\right)=2s.
\ee
\end{remark}

In terms of the Riemann sphere $\C\P^1$, we have that the (holomorphic) line bundle $\mathcal{O}(2s)\to \C\P^1$ has Euler characteristic $\chi(O(2s))=2s$. Therefore the spin $s$ bundle $\spi{s}$ must be isomorphic to the smooth complex line bundle underlying $\Phi^*\mathcal{O}(-2s)$ because $\Phi$ reverses the orientation, or, equivalently, $\Phi^*\overline{\mathcal{O}(2s)}$.  The latter is a holomorphic bundle with respect to the holomorphic structure on $S^2$. From this we see that Definition \ref{def:intrispin} can be extended to all (and not more) $s\in \frac12 \Z=\pm\frac{1}{2},\pm 1, \pm\frac{3}{2},\dots$. 
\begin{defi}
The \emph{spin s bundle} is the complex line bundle $\spi{s}\to S^2$ defined as
\be 
\spi{s}:=\Phi^*\overline{\mathcal{O}(2s)}, \quad \forall s\in\frac{1}{2}\Z.
\ee
\end{defi}
\begin{prop}\label{prop:intrispin}
Let $s\in \frac12 \N$, then the spin $s$ bundle is the $2s$ tensor power of the spin $\frac12$ bundle.
\bega 
\spi{s}&:=\left(\spi{\frac12}\right)^{\otimes 2s},
\\
\spi{-s}&:=\left(\spi{\frac12}\right)^{\otimes (-2s)}=\left(\left(\spi{\frac12}\right)^*\right)^{\otimes 2s}=\left(\spi{-\frac12}\right)^{\otimes 2s}.
\eega
\end{prop}
\subsection{$\sud$ is the radius $2$ sphere bundle of $\spi{\frac12}$}
\begin{defi}
Let $L\to M$ be a complex line bundle, endowed with an Hermitian  norm $\|\cdot\|$. Then its radius $r>0$ sphere bundle is the circle bundle $rS(L)\to M$
\be 
rS(L):=\{v\in L\colon \|v\|=r\}.
\ee
\end{defi}
\begin{remark}
If $L$ is a holomorphic line bundle over a K\"ahler Riemann surface, then the total space $rS(L)$ inherits a Riemannian metric, via the Chern connection of $L$, which makes $S(L)\to M$ a Riemannian submersion.
\end{remark}
The complex line bundle $\spi{\frac12}=\mathcal{O}(-1)\to\C\P^1$ is defined as
\be 
\mathcal{O}(-1)= \{(\ell, h)\in \S\times \C^2\colon h\in\ell\}\xrightarrow{\pi_1} \S.
\ee
Since the fiber over $\ell\in\S$ is $\ell\subset \C^2$, the bundle $\mathcal{O}(-1)$ is called the \emph{tautological} bundle. Notice also that by restricting to the complement of the zero section $\S\subset \mathcal{O}(-1)$, we get a tautological diffeomorphism
\bega\label{eq:tautmap}
\q \colon\H\-\{0\} &\xrightarrow{\diffeo} \mathcal{O}(-1)\- \S
\\
h=h(\a,\beta) &\mapsto \q_h:=\left(\frac{\a}{\beta},\begin{pmatrix}
\a \\ \beta
\end{pmatrix}\right).
\eega
In particular, this map transports the metric $g_\H$ into an Hermitian bundle metric on $\spi{\frac12}=\mathcal{O}(-1)$, such that the restriction of $\q$  to $\sud=2S^3\subset \H$, is an isomorphism of principal $2S^1$ bundles. Thus,
according to the \emph{polar bear} orientation on $\sud$, the map $\tau$ descends to an orientation preserving isometry of the total spaces:
\begin{prop} We have an isomorphism of complex line bundles
\be\label{eq:spherebu1}
\sud\iso 2S\left(\spi{\frac12}\right)\to S^2,
\ee
where the projection map $\sud \to S^2$ is given by the action on the north pole $e_3=\Phi^{-1}(\infty)\in S^2$, as in the diagram \eqref{eq:diagroups}.
\end{prop}
\begin{remark}
The total space of $\mathcal{O}(-1)$ is, by definition, the K\"ahler manifold obtained as the blow-up of $\C^2$ at the point $0$. It can be seen that then, in the category of smooth manifolds, $O(-1)$ is diffeomorphic to $\C\P^2\-\{pt\}$.
\end{remark}
\subsection{Spin weighted functions}
Let $p\in \S$ and $v\in \mathcal{T}_p^{\otimes\frac12}\-\{0\}$, then the fiber over $p$ of $\spi{s}$ is 
\bega
\mathcal{T}_p^{\otimes\frac12}&=\left\{\sum_{i}v_1^i\otimes\dots\otimes v_{2s}^i \text{, s.t. } v_j^i\in \spi{\frac12}_p\right\}=\{z\cdot v\otimes\dots \otimes v\colon z\in\C\}.
\eega
When $\xi$ changes: $\xi'=w\xi$, the vector $v^{\otimes 2s}=v\otimes\dots \otimes v$ changes accordingly to:
\be 
(v')^{\otimes 2s}=w^{2s} v^{\otimes 2s}.
\ee
When $s<0$, the above description still makes sense, if $v^{\otimes (-2s)}\in \spi{-\frac12}$ is defined as the linear form $\spi{s}\to \C$ such that $\langle v^{\otimes (-1)},v\rangle=1$, with respect to the duality pairing $\spi{-s}=\spi{s}^*$.
\begin{remark}\label{rem:cozione}
Notice that the coordinates of an element $\tau=z v^{\otimes 2s}=z'(v')^{\otimes 2s}\in \spi{s}$ have spin weight $=-s$:
\be 
z'=w^{-2s}z,
\ee
indeed ``the coordinates of vectors are covectors, hence they belong to the dual bundle''.
\end{remark}
\begin{remark}
In the book \cite[p. 287]{libro} we see that the transition functions for the bundle $\spi{s}$ are 
\be \label{eq:tranz}
f_{R_2}(x)=\exp(is\psi_{R_2 R_1})f_{R_1}(x),
\ee
where $\psi_{R_1 R_2}$ is the  angle between $\frac{\de}{\de \f_{R_1}}$ and $\frac{\de}{\de \f_{R_2}}$, measured in the usual way, from the outside of the sphere, see Remark \ref{rem:polar}.
\be 
\frac{\de}{\de \f_{R_2}}=e^{-i\psi_{R_2 R_1}}\frac{\de}{\de \f_{R_1}}.
\ee
Therefore the rule \eqref{eq:tranz} is equivalent to the transition rule for $(TS^2)^{\otimes s}=\spi{s}$, for any $s\in\N$:
\be 
f_{R_2}(x)\left( \frac{\de}{\de \f_{R_2}}\right)^{\otimes s}=f_{R_1}(x)\left( \frac{\de}{\de \f_{R_1}}\right)^{\otimes s}.
\ee
\end{remark}
\begin{defi}
We define a Hermitian bundle metric on $\spi{s}$ such that for any $g\in\sud$ we have
\be 
\|(\q_g)^{\otimes 2s}\|=\frac{1}{|s|};
\ee
 for any $s\in\frac12 \Z$, where $\tau\colon \sud\to \spi{\frac12}$ is defined as in \eqref{eq:tautmap}.
\end{defi}
In this way, the map $\q^{\otimes 2s}$ is a $2|s|$-fold covering of the radius $\frac{1}{|s|}$ circle bundle of $\spi{s}$.
\be \label{eq:spherebu2}
\q^{\otimes 2s}\colon \sud\to \frac{1}{|s|}S(\spi{s})=\left\{v\in\spi{s}\colon \|v\|=\frac{1}{|s|}\right\}.
\ee
The above map is a principal bundle with respect to the right action of the cyclic subgroup of order $2|s|$ generated by the element $g_3\left(\frac{2\pi}{s}\right)\in \sud$, which is the finite group
\bega 
\sqrt[2s]{1}&:=\{h\in\H\colon h^{2s}=1\}
\\
&=
 \left\{g_3\left(\frac{2\pi}{s}\right),g_3\left(\frac{2\pi}{s}2\right),\dots, g_3\left(\frac{2\pi}{s}(2|s|-1)\right),g_3\left(\frac{2\pi}{s}2|s|\right)\right\}
 \cong \Z_{2|s|}.
\eega
By declaring the map $\eqref{eq:spherebu2}$ to be a Riemannian covering (i.e. a covering that is also a Riemannian submersion), we can define a metric on the total space of $\frac{1}{|s|}S(\spi{s})$. We define the resulting Riemannian manifold as
\be 
\sudo{2s}:=\sud/\sqrt[2s]{1}\iso \frac{1}{|s|}S(\spi{s}).
\ee 
\begin{remark}
As Riemannian manifolds, there is no difference between $\sudo{2s}$ and $\sudo{-2s}$, but they have opposite orientations, since they are the total spaces of a pair of circle bundles over $S^2$ that are dual to each other.
\end{remark}
\begin{remark}
By definition, $\sudo{2s}\to \S$ is a Riemannian circle bundle having fibers of length $\frac{2\pi}{|s|}$. In particular $\sudo{2}\cong SO(3)$.
\end{remark}
We can in fact extend the commutative diagram in \eqref{eq:diagroups} to every $s\in\frac12 \Z$:
\be \label{eq:extdiagroups}
\begin{tikzcd}
2S^1 \arrow[r, hook] \arrow[d, ":s"] & SU(2) \arrow[d, ":s"] \arrow[r, "\pr"] & \overline{\C\P^1} \arrow[d, "\cong" description] \\
\frac1s S^1 \arrow[r, hook]       & \sudo{2s} \arrow[r, " "]                    & S^2                                      
\end{tikzcd}
\ee

Let $\sigma\colon S^2\to \spi{s}$ be a section. Then, obviously,  for any point $p\in S^2$ and $v\in \spi{\frac12}$, we have
\be 
\sigma(p)=z_\sigma(p,v) v^{\otimes 2s},
\ee
for some $z_\sigma(p,v)\in\C$. A convenient way to understand this $z_\sigma(p,v)$ is to observe that for any such $p,v$
 there exists a unique $g\in \sud$ such that $g\cdot \infty=p$ and $\q_g=v$. 
It follows that the section $\sigma\colon S^2\to \spi{s}$ is uniquely determined by a function $F_\sigma\colon \sud\to \C$ such that
\be \label{eq:idspirule}
\sigma (g\cdot \infty)= F_\sigma (g)\left(\q_g\right)^{\otimes 2s}.
\ee
It is easy to see that a function $F\colon \sud\to \C$ is associated with a section $\sigma$ of $\spi{s}$ if and only if
\be\label{eq:Fspinrule}
F(g \cdot g_3(\psi))=F(g)e^{-is\psi},
\ee
for any $\psi\in \R$. 
Thus, we have the following well known characterization of spin weighted functions, see \cite{BR13,eth}.
\begin{thm}\label{thm:pullback}
Sections of $\spi{s}$ are in bijections with functions $F\colon \sud\to \C$ that satisfy the rule \eqref{eq:Fspinrule}, via the identity \eqref{eq:idspirule}. We say that $F_\sigma$ is the \emph{pullback} of $\sigma$ (see \cite{BR13}) and that $F$ has \emph{right spin$=-s$}.
\end{thm}
\begin{remark}
This change of sign in the spin weight is explained by the fact that $F_\sigma(g)$ is actually a function that expresses the \emph{coordinates} (see Remark \ref{rem:cozione}) of $\sigma$ in the trivialization of the bundle $\spi{s}$ determined by $g$.
\end{remark}
\begin{cor}
Sections of $\spi{s}$ are (particular) functions on $\sudo{2s}$.
\end{cor}
\begin{defi}\label{def:purespin}
We denote the set of all smooth functions on $\sud$ with \emph{right spin} $=s\in\frac12 \Z$ as $\mathcal{R}(s)$.
\be
\mathcal{R}(s):=\left\{F\in\mC^\infty(\sud)\colon F(g\cdot g_3(\psi))=F(g)e^{is\psi}\right\}.
\ee
Similarly, the set of functions with \emph{left spin} $=m\in\frac12 \Z$ is
\be 
\mathcal{L}(m):=\left\{F\in\mC^\infty(\sud)\colon F(g_3(\psi)\cdot g)=e^{im\psi}F(g)\right\}.
\ee
We say that a function $F\colon \sud \to \C$ has \emph{pure left spin} or \emph{pure right spin} if it belongs to some of the spaces $\mathcal{L}(m)$ or $\mathcal{R}(s)$, for some $m,s$. We say that $F$ has $\emph{pure spin}$ if it has both pure left spin and pure right spin.
\end{defi}
Theorem \ref{thm:pullback} says that
\be
\mathcal{R}(-s)=\mC^\infty\left(S^2\Big|\ \spi{s}\right).
\ee
\begin{remark}
In terms of the coordinates $\a,\beta$, the left and right multiplication by $g_3(\psi)$ are given by the following  identities: 
\be\label{eq:lralpha}
h(\a,\beta)\cdot g_3(\psi)=h(\a e^{i\frac{\psi}{2}},\beta e^{i\frac{\psi}{2}}); \qquad 
g_3(\psi)\cdot h(\a,\beta)=h(\a e^{i\frac{\psi}{2}},\beta e^{-i\frac{\psi}{2}}).
\ee
\end{remark}
\section{Wigner functions}\label{sec:wigner}
For any $\ell\in\frac12\N$, consider the Hilbert space $\mathcal{H}_\ell=\C[z_0,z_1]_{(2\ell)}$ as a subspace of $L^2(\sud)$ of complex dimension $2\ell+1$. The resulting Hilbert product corresponds (up to a constant factor) with the Bombieri-Weyl product for which an orthonormal basis is given by the rescaled monomials: for any $z=h(z_0,z_1)\in \sud$,
\be\label{eq:wignerfuncdef}
\psi^\ell_m(z):={{2\ell}\choose{\ell+m}}^\frac12z_0^{\ell+m}z_1^{\ell-m}, \quad m=-\ell,\ell +1,\dots ,\ell \in\frac12 \Z.
\ee
Then we define the extended Wigner function $\hat{D}^\ell_{m,s}\colon \H\to \C$ on any matrix $h=h(\a,\beta)\in\H$ by 
\be\label{eq:wignerdef}
\sum_{m=-\ell}^\ell\psi_{m}^\ell(z)\hat{D}^\ell_{m,s}(h)=\psi_{s}^\ell\left(h^{-1}\begin{pmatrix}
z_0 & -\overline{z_1}\\ z_1 &\overline{z_0}
\end{pmatrix}\right)=(\overline{\a}z_0+\overline{\beta}z_1)^{\ell+s}(-\beta z_0+\a z_1)^{\ell-s}{{2\ell}\choose{\ell-s}}^\frac12.
\ee
The (standard) Wigner function is the restriction to $\sud$:
\be 
\hat{D}^\ell_{m,s}|_{\sud}=D^\ell_{m,s}\colon \sud\to \C,
\ee
\be 
D^\ell\colon \sud\to U(2\ell+1); \qquad 
D^\ell(g)=\left(D^\ell_{m,s}(g)\right)_{-\ell,\le m,s \le \ell}.
\ee
We will also be interested in the columns of $D^\ell$, which we will denote as follows.
\be 
\cold{s}, \rowd{m} \colon \sud\to S^{2\ell}\subset \C^{2\ell+1}.
\ee
Moreover, we define the spaces
\be 
\coldsp{s}:=\text{span}_\C\{D^\ell_{m,s}\colon m=-\ell,\dots,\ell\}\subset \mC^\infty(\sud,\C);
\ee
\be 
\rowdsp{m}:=\text{span}_\C\{D^\ell_{m,s}\colon s=-\ell,\dots,\ell\}\subset \mC^\infty(\sud,\C);
\ee
\be 
\mathscr{D}^\ell:=\mathscr{D}^\ell_{\bullet,-\ell}+\dots +\mathscr{D}^\ell_{\bullet,\ell}\subset \mC^\infty(\sud,\C).
\ee
\subsection{Irreducible representations}
Let us consider the \emph{pull-back action} of $I\in SO(4)$ on $L^2(\sud)$, defined by: $F\mapsto I^{-*}F:=F\circ (I^{-1})$. Since $SO(4)=\{L_{g_1}\circ R_{g_2}\colon g_1,g_2\in\sud\}$, we can think of it as an action of $\sud\times \sud$\footnote{The group $SO(4)$ is not isomorphic to $\sud\times \sud$, but this will not be important for our purpose.}. We will use the convention of left actions: the pair $(g_1,g_2)\in \sud\times \sud$ acts as $L_{g_1}^{-*}\circ R_{g_2}^*$, where for any function $F\colon \sud\to \C$ 
\be 
 L_{g_1}^{-*}\circ R_{g_2}^*F (z)=F(g_1^{-1}zg_2).
\ee
 In particular, we have the two different actions of $\sud\times 1$ and $1\times \sud$ on $L^2(\sud)$ corresponding to the pull-back of left and right multiplications. We will refer to them as \emph{left} and \emph{right pull-back actions}. Moreover, by identifying the subgroup $K:=\{g_3(\psi)\colon \psi\in [0,4\pi]\}\subset \sud$ with $U(1)$, we see that the pull-back actions restrict to a unitary action of $U(1)\times U(1)$. Clearly, the left and right spin are related to this action, indeed $\mathcal{L}(m)\cap \mathcal{R}(s)$ are invariant subspaces.

The space $\mathcal{H}_\ell\subset L^2(\sud)$ is $\sud\times 1$ invariant\footnote{Notice that $\mathcal{H}_\ell$ is not invariant for the action of $1\times \sud$, because multiplication on the right would mix the variables $z_i$ and $\overline{z_i}$.}.
By definition, the Wigner matrices are the matrices that correspond to such unitary representation of $\sud$.
\begin{prop}
Any unitary irreducible representation of $\sud$, is equivalent to one and only one of the Wigner matrices $D^\ell\colon \sud \to U(2\ell+1)$ for some $\ell\in\frac12 \N$. In the sense of representation theory this means that the dual of $\sud$ is the set
\be 
\hat{\sud}=\{(D^\ell,\C^{2\ell+1})\}_{\ell\in\frac12 \N}.
\ee
\end{prop}
\begin{proof}
See \cite[Theorem 3.14]{libro}.
\end{proof}
By the Peter-Weyl theorem \cite{libro}, the space $L^2(\sud)$ splits as an orthogonal sum of the spaces of matrix coefficients $\matrsp$, which are irreducible spaces for the whole pull-back action, i.e. the action of $SO(4)$. The Peter-Weyl theorem asserts that, in fact, each of the spaces  $\matrsp$ splits again into the spaces of columns coefficients $\coldsp{s}$, which are irreducible for the left pull-back action of $\sud\times 1$ and give equivalent unitary representations. Finally, the space $\coldsp{s}$ is $U(1)\times 1$ invariant and thus it splits again into irreducible subspaces for such circle action. Since $U(1)$ is abelian, its irreducible representations are forced to be $1$-dimensional, thus we conclude that there exists an orthonormal basis of $\mathscr{D}_{\bullet,s}^\ell$ consisting of functions with pure left spin. Notice that the monomials are special as a basis of $\mathcal{H}_\ell$, in that $\psi^\ell_m\in \mathcal{L}(m)$. The consequence of this choice of basis is that the coefficients of the corresponding matrix representation have pure right spin, and precisely $D^{\ell}_{m,s}\in \mathcal{R}(-s)$:
\be  
\sum_{m}\psi^\ell_m(z)D^\ell_{m,s}(g\cdot g_3(\psi))=\psi^\ell_s\left( g_3(\psi)^{-1} g^{-1}z\right)=e^{-is\psi}\psi^\ell_s(g^{-1}z)=\sum_{m}\psi^\ell_m(z)D^\ell_{m,s}(g)e^{-is\psi}.
\ee 
Then, from unitarity $D(g)^\ell_{m,s}=\overline{D^\ell_{s,m}(g^{-1})}$, we deduce that $D^\ell_{m,s}\in\mathcal{L}(-m)\cap \mathcal{R} (-s)$, hence the orthonormal basis of pure spin functions for $\coldsp{s}$ is indeed given by the Wigner functions $D^\ell_{m,s}$. Moreover, by Schur's orthogonality relations (see \cite{libro}), we deduce that, with respect to the metric $\sud=2S^3$, we have
\be 
\|D^\ell_{m,s}\|^2_{L^2(\sud)}=\frac{\vol(\sud)}{2\ell+1}=\frac{16\pi^2}{2\ell +1}.
\ee
\begin{prop}\label{prop:dirrep} We have the following orthogonal decompositions of $L^2(\sud)$:
\be 
L^2(\sud)=\bigoplus_{\ell\in\frac12 \N} \matrsp=
\bigoplus_{\ell\in\frac12 \N}\bigoplus_{ s=-\ell,\dots,+\ell}\coldsp{s}=\bigoplus_{\ell\in\frac12 \N}\bigoplus_{ s=-\ell,\dots,+\ell}\bigoplus_{ m=-\ell,\dots,+\ell} \C D^{\ell}_{m,s}.
\ee
\begin{enumerate}
\item The space $\matrsp$ is irreducible for $SO(4)$;
\item The space $\coldsp{s}$ is irreducible for the action of $\sud\times U(1)$;
\item The space $\rowdsp{m}$ is irreducible for the action of $U(1)\times \sud$;
\item $D^\ell_{m,s} \in \mathcal{L}(-m)\cap \mathcal{R}(-s)$ has left spin $=-m$ and right spin $=-s$.
\end{enumerate}
\end{prop}
\begin{remark}
In the literature, one might find different conventions for the definitions of $D^\ell_{m,s}$. 
Anyway, the above decomposition characterizes the functions $D^\ell_{m,s}$ up to multiplication by a phase. The additional requirement that $D^\ell_{m,s}(g_2(\h))\in \R$, which is true in the case of this paper, determines them uniquely up to a sign. 
\end{remark}
We report an additional property of the functions $D^\ell_{m,s}$.
\begin{prop}\label{prop:otherD}
For any $m,s,\ell$ and $g\in\sud$, we have
\be 
\overline{D^\ell_{m,s}(g)}=D^\ell_{m,s}(\overline{g})=D^\ell_{-m,-s}(g)(-1)^{2\ell+m-s}.
\ee
\end{prop}
\begin{proof}
For the first identity, observe that $\overline{\psi^\ell_s(z)}=\psi^\ell_s(\overline{z})$, so that by taking the conjugate in 	\eqref{eq:wignerdef} we have
\be  
\sum_{m}\psi^\ell_m(\overline{z})\overline{D^\ell_{m,s}(g)}=\psi^\ell_s\left( \overline{g}^{-1}\overline{z}\right)=\sum_{m}\psi^\ell_m(\overline{z})D^\ell_{m,s}(\overline{g})e^{-is\psi}.
\ee 
To prove the second identity, we use the fact that for all $z=h(z_0,z_1)\in\H$, we have
\be 
\psi^\ell_s(g_2(\pi)^{-1}z)=\left(z_1\right)^{\ell+s}\left(-z_0\right)^{\ell-s}{{2\ell}\choose{\ell+s}}^\frac12=(-1)^{\ell-s}\psi^\ell_{-s}(z).
\ee
Thus, $D^\ell_{m,s}(g_2(\pi))=\delta_{-m,s}(-1)^{\ell+m}$. From this, by using the representation of $g\in \sud$ via Euler angles, as in \eqref{eq:Eule}, $g=g_3(\phi) g_2(\h) g_3(\psi)$ and the fact that $D^\ell_{m,s}$ have pure left and right spin, we obtain
\be  
D^\ell_{m,s}(\overline{g})=e^{im\phi}D^\ell_{m,s}(g_2(\pi+\h-\pi))e^{is\psi}=e^{im\phi}\sum_{a,b}D^\ell_{m,a}(g_2(\pi))D^\ell_{a,b}(g_2(\h))D^\ell_{b,s}(g_2(-\pi))e^{is\psi}=\dots
\ee 
\be 
\dots =e^{im\phi}\sum_{a,b}\delta_{-m,a}(-1)^{\ell+m}D^\ell_{a,b}(g_2(\h))\delta_{s,-b}(-1)^{(\ell-s)}e^{is\psi}=\dots
\ee
\be
\dots=
(-1)^{2\ell+m-s}e^{im\phi}D^\ell_{-m,-s}(g_2(\h))e^{is\psi}=(-1)^{2\ell+m-s}D^\ell_{-m,-s}(g).
\ee
\end{proof}
\subsection{Laplacians}
The above decomposition of $L^2(\sud)$ can be also seen as a consequence of the decomposition of the  Laplacian into a vertical and a horizontal part, in the sense of \cite{Borg}.  Indeed the map
\be 
\pr\colon\sud\to S^2
\ee
is a Riemannian submersion with totally geodesic fibers (its fibers are big circles in $2S^3$) and therefore the Laplace-Beltrami operator can be written as a sum of two commuting self-adjoint operators:
\be\label{eq:laplsplit}
\Delta_{\sud}=\Delta_h+\Delta_v.
\ee
Following \cite{Borg}, the vertical Laplacian $\Delta_v$ is defined as the Laplace-Beltrami operator of the fibers of $\pr$. In our conventions, it corresponds to second derivative with respect to the Euler angle $\psi$:
\be 
\Delta_v F (g):=\frac{d^2}{d\psi^2}\Big|_{\psi=0}F(g\cdot g_3(\psi));
\ee
while $\Delta_h$ is defined by the identity \eqref{eq:laplsplit}.
Clearly, functions with pure right spin are eigenfunctions of $\Delta_v$, indeed 
\be 
\ker(\Delta_v+s^2)=\mathcal{R}(s)+\mathcal{R}(-s).
\ee

What's most remarkable about the decomposition \eqref{eq:laplsplit} is that $\Delta_h$ and $\Delta_v$ commute (this is proved in \cite[Theorem 1.5]{Borg}). This implies that there exists an orthogonal decomposition of $L^2(\sud)$ into common eigenspaces of $\Delta_{\sud},\Delta_v$ and $\Delta_h$ (it is sufficient to find a Hilbert basis of common eigenfunctions for $\Delta_{\sud}$ and $\Delta_v$, that is \cite[Proposition 1.3]{Borg}).

By the following observation, we conclude that such common eigenspaces are exactly the spaces of column coefficients of the Wigner matrices.
\begin{prop}\label{thm:Disharmonic}
$\hat{D}^\ell_{m,s}\colon \H=\R^4\to \C$ is a real harmonic polynomial of degree $2\ell$. Therefore, under the normalization 
\be 
\phi^{\ell}_{m,s}:=\frac{\sqrt{2\ell+1}}{4\pi}D^\ell_{m,s},
\ee
the collection  $\phi^\ell_{m,s}$, for all $\ell,m,s\in \frac12 \N$ with $-\ell\le m,s \le \ell$, form an orthonormal basis of $L^2(\sud)$ of spherical harmonics on $\sud=2S^3$,  i.e. eigenfunctions of $\Delta_{\sud}$, with eigenvalue $-\frac12\ell(2\ell+2)$.
\end{prop}
\begin{remark}\label{rem:spharm}
To obtain the so called \emph{spin weighted spherical harmonics} $Y^\ell_{m,s}$ one has to normalize $D^\ell_{m,s}$ as a section of $\spi{s}$, in which case the $L^2$ norm is given by integrating on the sphere $S^2$. Thus they are obtained by dividing $\phi^\ell_{m,s}$ by the square root of the length of the fiber of $\sud\to S^2$, which is $4\pi$:
\be 
Y^\ell_{m,s}=\sqrt{\frac{2\ell+1}{4\pi}}D^\ell_{m,s}.
\ee
\end{remark}
\begin{proof}
It is clear by the definitions \eqref{eq:wignerdef} that the Wigner function $D^\ell_{m,s}$ is a homogeneous polynomials with complex coefficients of degree $2\ell$ in the variables $\a,\overline{\a},\beta,\overline{\beta}$. It follows that its real and imaginary parts are real homogeneous polynomials of degree $2\ell$ in the real coordinates of $\H=\R^4$.
The Laplacian of $\H=\C^2$, with the metric $\langle \cdot, \cdot \rangle_\H=4\langle \cdot, \cdot \rangle_{\R^4}$ is the operator
\be 
\Delta_\H=\frac14\Delta_{\R^4}=\frac{1}{4}\frac{\de}{\de \a}\frac{\de}{\de \overline{\a}}+\frac{1}{4}\frac{\de}{\de \beta}\frac{\de}{\de \overline{\beta}}\colon \mC^\infty(\H,\C)\to \mC^\infty(\H,\C).
\ee
By noticing that $\psi^\ell_m(z)$ is holomorphic, i.e. it can be written as a polynomial in $z_0$ and $z_1$, where $z=h(z_1,z_2)$, we can conclude by proving the following Lemma. 
\begin{lemma}
If $F(z)=f(z_0,z_1)$ is holomorphic, then the polynomial $\H\ni \a,\beta\mapsto F(h(\a,\beta)^{-1}z)\in\C$ is harmonic.
\end{lemma}
Let $z=h(z_0,z_1)$ and $h=h(\a,\beta)$,
\bega
\phi_z(\a,\beta):=F(h(\a,\beta)^{-1}h(z_0,z_1))=F\left(\begin{pmatrix}
\overline{\a} & -\beta \\
\overline{\beta} & \a
\end{pmatrix}\begin{pmatrix}
z_0 & -\overline{z_1} \\
z_1 &\overline{z_0} 
\end{pmatrix}\right)=f(\overline{\a}z_0-\beta z_1,\overline{\beta}z_0+\a z_1).
\eega
Then
\bega 
4\Delta_{\H}\phi_z(\a,\beta)&=\frac{\de}{\de \a}\frac{\de}{\de \overline{\a}}+\frac{\de}{\de \beta}\frac{\de}{\de \overline{\beta}} f(\overline{\a}z_0-\beta z_1,\overline{\beta}z_0+\a z_1)=
\\
&=\frac{\de}{\de \a}
\left(
	z_0\left(\frac{\de f}{\de z_0}\right)(\overline{\a}z_0-\beta z_1,\overline{\beta}z_0+\a z_1)\right)
	+
	\frac{\de}{\de \beta}\left(z_0\left(\frac{\de f}{\de {z_1}}\right)(\overline{\a}z_0-\beta z_1,\overline{\beta}z_0+\a z_1)
	\right)
	\\
	&=
z_0z_1\left(\frac{\de^2 f}{\de z_0 \de z_1}\right)(\overline{\a}z_0-\beta z_1,\overline{\beta}z_0+\a z_1)
	-z_0z_1\left(\frac{\de^2 f}{\de {z_1}\de z_0}\right)(\overline{\a}z_0-\beta z_1,\overline{\beta}z_0+\a z_1)=0.
\eega
\end{proof}
\begin{cor} The spaces
$\coldsp{s}$
are common eigenspaces of $\Delta_{\sud},\Delta_v$ and $\Delta_h$ with eigenvalue:
\bega
\Delta_{\sud} {|_{\coldsp{s}}}&=-\frac12\ell(2\ell+2),
\\
\Delta_v  {|_{\coldsp{s}}}&=-s^2,
\\
\Delta_h  {|_{\coldsp{s}}}&=-\frac12\ell(2\ell+2)+s^2=-(\ell - s)(\ell +s+1)-s.
\eega
\end{cor}
In terms of the Laplacian defined in terms of the spin raising and spin lowering operators, it was proved in by Newman and Penrose\cite{NP66} (see also\cite{GM10, libro, eth}) that the Wigner functions are eigenfunctions of the self-adjoint operator  $\eth\overline{\eth}\colon \mathcal{R}(-s)\to \mathcal{R}(-s)$:
\be 
 \eth\overline{\eth}D^{\ell}_{m,s}=-(\ell-s)(\ell+s+1)D^\ell_{m,s}.
\ee
This proves theorem \eqref{eq:decla} and \eqref{eq:kuwa}:
\be 
\Delta_{\sud}=\eth\overline{\eth}+\frac{d^2}{d\psi^2}+\sqrt{\frac{d^2}{d\psi^2}}; \quad \Delta_h=\eth\overline{\eth}+\sqrt{\frac{d^2}{d\psi^2}}.
\ee
\subsection{Random Wigner matrices}\label{sec:randomvecsph}
In this section we will think the group $\sud$ as a probability space, with the volume measure $\mu$ of $2S^3$ normalized to $1$. Using the notation of random elements, we will write $\gamma$ for a random element of $\sud$ such that for every measurable subset $A\subset \sud$ and measurable function $F\colon \sud\to \R$, 
\be 
\P\{\gamma\in A\}=\frac{\vol(A)}{\vol(\sud)}; \qquad \E\{F(\gamma)\}=\frac{1}{\vol(\sud)}\int_{\sud}F(g)d\mu(g).
\ee
In other words, $\gamma$ is a random variable with values in $\sud$ whose law is the Haar probability measure of $\sud$.
This is just a convenient notation.

Let $\langle v,w \rangle=\overline{v}^Tw$ be the Hermitian product of $\C^{2\ell+1}$, then the real part $\Re\langle w,v \rangle$ is the standard Euclidean product, under the identification $\C^{2\ell+1}=\R^{2(2\ell+1)}$. The unitarity of the representation $D^\ell$ means that the matrices $D^\ell(g)$ act on $\C^{2\ell+1}$ as unitary operators. Thus, in particular, they preserves all spheres $rS^{4\ell+1}$.

By means of the homomorphism map  $D^\ell\colon \sud\to U({2\ell+1})$, we get a collection of random matrices $D^\ell(\gamma)$, for all $\ell\in\frac12 \N$ such that
\be\label{eq:Dlinv} 
\left(D^\ell(\gamma)\right)_{\ell\in\frac12\N}
\law
\left(D^\ell(g)D^\ell(\gamma)\right)_{\ell\in\frac12\N},
\ee
and for each $v\in S^{4\ell+1}$, there is a collection of random vectors in $S(\C^{2\ell+1})=S^{4\ell+1}$, for all $\ell\in\frac12\N$.
\begin{defi}
$\gamma^\ell_v:=D^\ell(\gamma)v$ for every $v\in S^{4\ell+1}.$
\end{defi}
The random vector $\gamma_v$ is supported on the orbit of $v$ under the action of $\sud$, and its law is invariant by the unitary transformations of the form $D^\ell(g)$, for every $g\in\sud$. By the theory of smooth group actions (see \cite{kirillov}), this is a smooth submanifold in $S^{4\ell+1}$ and the canonical map
\be\label{eq:quotientistrp}
\sud /H^\ell(v)\hookrightarrow O^\ell(v):=\{D^\ell(g)v\colon g\in\sud\}
\ee
is an embedding and a principal bundle, 
where $H^\ell(v):=\{g\in\sud\colon D^\ell(g)v=v\}$ is the isotropy subgroup of $v$. Notice that this implies that the law of $\gamma_v$ is equivalent to that of $D^\ell(g)\gamma_v$ for any $g\in\sud$, therefore it coincides with the normalized Riemannian volume measure of $O^\ell(v)\subset S^{4\ell+1}$. 
\begin{prop}\label{prop:Schurorth}
The random vectors $\gamma_v$ are centered: $\E\{\gamma_v\}=0$ for all $\ell\neq 0$, while $\gamma_v=v$ is constant if $\ell=0$. Their correlation is characterized by the following identities. For every $v,w\in\C^{2\ell+1}$ and $v',w'\in\C^{2\ell'+1}$, we have
\be\label{eq:schur}
\overline{w}^TK^\ell(v,v')w:=\E\left\{\overline{\langle \gamma_v,w\rangle} \langle \gamma_{v'},w'\rangle\right\}=\delta_{\ell,\ell'}\frac{\overline{\langle v,v'\rangle} \langle w,w'\rangle}{2\ell+1}.
\ee
\end{prop}
\begin{proof}
Clearly, the vector $e=\E\{\gamma_v\}$ generates an invariant subspace of $\C^{2\ell+1}$, thus $e=0$, because of the irreducibility of $D^\ell$, except in the case $\ell=0$, where $D^0(g)=1$.
The identities \eqref{eq:schur} are just a reformulation of the Schur's orthogonality relations, see \cite{libro}. We repeat the proof here, since this result will be of fundamental importance for the rest of the paper.
Let us define $K(v,v')\colon \C^{2\ell'+1}\to \C^{2\ell+1}$ to be the (complex) linear operator
\be 
K(v,v')(w)=\E\{\gamma_v (\overline{\gamma_{v'}})^T\}w.
\ee
Then $K(v,v')$ intertwines the two unitary representations $D^\ell$ and $D^{\ell'}$:
\bega
K(v,v')(D^{\ell'}(g)w)&=\E\left\{\gamma_v\left((\overline{\gamma_{v'}})^TD^{\ell'}(g)w\right)\right\}
=
\E\left\{\gamma_v \left(w^TD^{\ell'}(g)^T \overline{\gamma_{v'}}\right)\right\}
=
\E\left\{\gamma_v w^T\overline{D^{\ell'}(g^{-1}) }\overline{\gamma_{v'}}\right\}
=\dots
\eega
and using the $D$-invariance of the collection $\gamma_{v}$, we have
\bega
\dots=
\E\left\{D^\ell(g)\gamma_v w^T\overline{D^{\ell'}(g^{-1}) }\overline{D^{\ell'}(g)\gamma_{v'}}\right\}
=
D^\ell(g) \E\left\{\gamma_v \left(\overline{\gamma_{v'}}\right)^T\right\}w=D^\ell(g)K(v,v')w.
\eega
By Schur's lemma (see \cite{libro}), the operator $K(v,v')$ is certainly equal to $0$ if $\ell\neq\ell'$, because it intertwines two non equivalent irreducible representations. If $\ell=\ell'$ and $i=0$, then $K(v,v')$ is an endomorphism of $\C^{2\ell+1}$, thus it has an eigenvalue. By Schur's lemma (or rather, its proof) again, the relative eigenspace is invariant, hence it must be the whole space, so that we conclude that there exists $\lambda\in\C$ such that
\be 
K(v,v')=\lambda \mathbb{1}.
\ee
Now, $\lambda$ can be computed, by taking the trace.
\be 
\lambda (2\ell+1)=\textrm{tr}(K(v,v'))=\E\left\{\textrm{tr}(\gamma_v (\overline{\gamma_{v'}})^T)\right\}=\E\left\{\textrm{tr}((\overline{\gamma_{v'}})^T\gamma_v)\right\}=\E\left\{\langle\gamma_{v'},\gamma_v\rangle\right\}=\langle v',v\rangle.
\ee
\end{proof}
\begin{thm}\label{thm:Dcorr}
The collection of random matrices $\left(D^\ell(\gamma)\right)_{\ell\in\frac12 \N}$ are pairwise uncorrelated and their correlation is characterized by the following identities:
\be\label{eq:Dcor}
\E\left\{D^\ell_{m,s}(\gamma)\left(\overline{D^{\ell'}_{m',s'}(\gamma)}\right)\right\}=\frac{\delta_{\ell,\ell'}\delta_{s,s'}\delta_{m,m'}}{2\ell+1};
\quad
\E\left\{D^\ell_{m,s}(\gamma)D^{\ell'}_{m',s'}(\gamma)\right\}=\frac{\delta_{\ell,\ell'}\delta_{-s,s'}\delta_{-m,m'}  (-1)^{2\ell -m+s}}{2\ell+1}.
\ee
\end{thm}
\begin{proof}
The first identities in \eqref{eq:Dcor} follows directly from Theorem \ref{prop:Schurorth}, when $ w=e_m, v=e_s\in\C^{2\ell+1}$ and $ w'=e_{m'}, v'=e_{s'}\in\C^{2\ell'+1}$ are the vectors of the canonical basis. The second identities are deduced from the first, via Proposition \ref{prop:otherD}
\bega
\E\left\{D^\ell_{m,s}(\gamma)D^{\ell'}_{m',s'}(\gamma)\right\}
=
\E\left\{D^\ell_{m,s}(\gamma)\left(\overline{D^{\ell'}_{-m',-s'}(\gamma)}\right)\right\}(-1)^{2\ell +m'-s'}.
\eega
\end{proof}
\begin{remark}
Notice that each of the random functions $D^\ell_{m,s}(\gamma)$ is circularly symmetric, for all $(m,s)\neq (0,0)$, while $D^\ell_{0,0}(\gamma)\in\R$. 
However, this is not true for the law of the whole collection of random variables $(D^\ell_{m,s}(\gamma))_{\ell,m,s}$, otherwise the left hand side in the second equation \eqref{eq:Dcor} would be always $=0$.
\end{remark} 
\subsection{Orbits}
The irreducibility of the representation $D^\ell\colon \sud\to U(2\ell+1)$, can be equivalently expressed by saying that for each $v\in S^{4\ell+1}$ the orbit $O_v^\ell$ spans the whole space (it is a consequence of Theorem \ref{prop:Schurorth}):
\be 
\text{span}(O^\ell_v)=\C^{2\ell+1}.
\ee
In fact, a consequence of Proposition \ref{thm:Disharmonic} is that the same property holds for any non negligible subset of $O^\ell_v$.
\begin{thm}\label{thm:span}
Let $A\subset O^\ell(v)$ be a measurable subset such that $\vol(A)\neq 0$, then $\text{span}(A)=\C^{2\ell+1}$.
\end{thm} 
\begin{proof}
Let $w,v\in\C^{2\ell+1}$, we will show that if $\text{span}(A)$ is contained in $w^\perp$, then $\P\{\gamma_v\in A\}=0$. Let $f\colon \sud\to \R$ be the function $f\colon g\mapsto \Re \left(\overline{w}^TD^\ell(g)v\right)\in\R$. By Theorem \ref{thm:Disharmonic}, $f$ is an eigenfunction of $\Delta_{\sud}$, therefore its nodal set $f^{-1}(0)$
 has Hausdorff dimension $n-1$ (see \cite{chavel1984eigenvalues}). Thus we conclude:
 \be 
\P\{\gamma_v\in A\}\le \P\{\gamma_v\in w^\perp\}=\vol(f^{-1}(0))=0.
\ee
\end{proof}
Despite Theorem \ref{thm:span}, it might very well be that the homeomorphism type and even the dimension of the orbits $O^\ell(v)$ are different for different choices of $v\in S^{4\ell+1}$. This is what happens for the columns of the matrix $D^\ell$, which parametrize the orbits of the canonical basis, i.e. the vectors $e_s$, $s=-\ell,\dots \ell$: 
\be 
\gamma_{e_s}=\cold{s}(\gamma).
\ee
These orbits are special, in that they correspond to the orbits of the monomial basis, under the identification $\C^{2\ell+1}\cong \mathcal{H}_\ell$. 
\begin{thm}
The following things are true.
\begin{enumerate}
\item The map \eqref{eq:quotientistrp} induces an embedding $\sudo{2s}\diffeo O^\ell(e_s)\subset S^{4\ell+1}$ for all $s\in\{-\ell,\dots \ell\}\-\{0\}$.
\item For $s=0$ there are two cases: if $\ell\in 2\N$, then $S^1\diffeo O^\ell(e_0)$; while $S^2\diffeo O^\ell(e_0)$ otherwise.
\item The above map descends to a smooth map of the sphere $S^2$ in $\C\P^{2\ell}$, which is an embedding, except in the case $\ell\in 2\N$ and $s=0$. 
\item For almost every $v\in S^{4\ell+1}$, the orbit $O^\ell(v)$ is diffeomorphic to $\sud$ when $\ell\notin \N$ and to $SO(3)$ when $\ell\in\N$. 
\end{enumerate}
\end{thm}
\begin{proof}
To prove $(1)$ and $(2)$ it is sufficient to compute the isotropy group of $e_s$.
Let us take $\C^{2\ell+1}=\mathcal{H}_\ell$, with $e_s$ corresponding to the monomial $\psi^\ell_s$. Then the isotropy group of $e_s$ is the set of matrices $h(\a,\beta)\in\sud$ such that
\be\label{eq:istropid}
(\overline{\a}z_0+\overline{\beta}z_1)^{\ell+s}(-\beta z_0+\a z_1)^{\ell-s}=z_0^{\ell+s}z_1^{\ell-s}.
\ee
By evaluating the above identity of polynomials at the points $(z_0,z_1)=(1,0), (0,1)$, we see that either $\a=0$ or $\beta=0$. Now, observe that with  $\alpha=0$ (and consequently, $|\beta|=1$), the equation \eqref{eq:istropid} becomes
\be 
(-1)^{\ell-s}(\beta)^{-2s}z_1^{\ell+s}z_0^{\ell-s}=z_0^{\ell+s}z_1^{\ell-s}.
\ee
This admits solutions only if $s=0$ and $\ell\in 2\N$.
In the case $\beta=0$, the equation to solve is
\be 
\a^{-2s}=1,
\ee
whose set of solution is, by definition, the subgroup $\sqrt[2s]{1}\subset \sud$. In synthesis we just proved that for all $s\in\{-\ell,\dots,\ell\}\-\{0\}$, the istropy subgroup of $e_s$ is
\be 
H^\ell(e_s)=\sqrt[2s]{1};
\ee 
while, for $s=0$, we have two cases:
\be 
H^\ell(e_0)=\left\{ h(e^{i\frac{\psi}2},0)\colon \psi\in [0,4\pi] \right\}, \text{ if $\ell\notin 2\N$}; 
\ee
\be 
H^\ell(e_0)=\left\{h(e^{i\frac{\f+\psi}2},e^{i\frac{-\f+\psi}2})\colon \f\in [0,2\pi],\psi\in [0,4\pi] \right\}, \text{ if $\ell\in 2\N$}.
\ee

Point $(3)$ follows from the fact that $\cold{s}\colon \sud\to S^{4\ell+1}$ is a spin $s$ function, hence it maps the fibers of the circle bundle $\sudo{s}\to S^2$ to fibers of the Hopf fibrations $S^{4\ell+1}\to \C\P^{2\ell}$. 

Point $(4)$ is a consequence of the so called \emph{Principal Orbit theorem}. In one of its stronger forms, proved in \cite{Montgo56Exceptional}, it says that union of orbits that are not maximal, both in the senses maximal dimension and minimal isotropy group, form a subset of codimension $2$. Since being an embedding is an open condition, for any point $v$ close enough to $e_\frac12\in S^{4\ell+1}$ the orbit is $O^\ell(v)\diffeo \sud$, therefore it is has typical orbit type of the action. In the case $\ell\in\N$ then $s\in \N$ as well and thus every map $\cold{s}$ descends to a map $\sudo{2}=SO(3)\to S^{4\ell+1}$, so that we can repeat the previous argument, but for $e_{1}$.
\end{proof}
\section{D-invariance}\label{sec:dinva}
In this section we study collection of vectors in $\C^{2\ell+1}$ that are invariant under the action of Wigner matrices, with the purpose of applying our results to the spectral coefficients $a^\ell_{m,s}$ of the decomposition as in Theorem \ref{thm:main1}. To this end, let us introduce some terminology.
\begin{defi}
We say that a collection of random vectors $(v_i)_{i\in I}$ with $v_i\in \C^{2\ell_i+1}$ is \emph{strongly $D$-invariant} if their joint law is equivalent to that of the collection $(D^{\ell_i}(g)v_i)_{i\in I}$ for any $g\in \sud$.
\end{defi}
\begin{defi}
Two random vectors $V,V'\in \C^{N}$ are said to be \emph{$2$-weakly equivalent} if they have the same expectation: $\E\{V\}=\E\{V'\}$ and the same self-correlation matrices: $\E\{V\overline{V}^T\}=\E\{V'\overline{V'}^T\}$ and $\E\{VV^T\}=\E\{V'(V')^T\}$. In this case, we will write
\be 
V\twik V'.
\ee
\end{defi}
\begin{defi}
We say that a collection $(V_i)_{i\in I}$ of random vectors $V_i\in \C^{2\ell_i+1}$ is \emph{$2$-weakly $D$-invariant} if for any $i,j$ and $g\in\sud$, we have
\be 
(V_i,V_j)\twik (D^{\ell_i}(g)V_i, D^{\ell_j}(g)V_j).
\ee 
\end{defi}
Given any collection $\mathcal{V}=(V_i)_{i\in I}$ of random vectors $V_i\in \C^{2\ell_i+1}$, there is an easy way to construct a strongly $D$-invariant one, simply by multiplying it by an independent random matrix $D^\ell(\gamma)$.
\begin{remark}
The collection $\mathcal{V}=(D^\ell(\gamma)V_i)_{i\in I}$ obtained in such way 
is always strongly $D$-invariant.
\end{remark}
In fact, it is almost tautological that any $D$-invariant collections is essentially of this form, since such operation can be seen as a projection on the space of strongly $D$-invariant probability measures. We will not enter into the details of this point of view, but we will give a concrete statement, to be precise.
\begin{thm}\label{thm:VDgamma}
A collection $\mathcal{V}=(V_i)_{i\in I}$ of random vectors $V_i\in \C^{2\ell_i+1}$ is strongly or $2$-weakly $D$-invariant, respectively, if and only if 
\be 
(V_i)_{i\in I}\law (D^\ell(\gamma)V_i)_{i\in I} \quad \text{ or }\quad  (V_i)_{i\in I}\twik (D^\ell(\gamma)V_i)_{i\in I},
\ee
where $\gamma\in\sud$ is a random element independent from $\mathcal{V}$, in the sense of section \ref{sec:randomvecsph}.
\end{thm}
\begin{proof}
Let us start from the case of strong $D$-invariance. Let $F_i\colon \C^{2\ell_i+1}\to \R$, for all $i\in I$, be any collection of measurable functions. Then, the strong $D$-invariance yields
\be\label{eq:EFi}
\E\{F_i(V_i)\}=\E\{F_i(D^\ell(\gamma)V_i)\}.
\ee
Therefore the joint distribution of the two collections $(V_i)_{i\in I}$ and $(D^\ell(\gamma)V_i)_{i\in I}$ are the same.
In the case of a $2$-weak $D$-invariant collection, we observe that equation \eqref{eq:EFi} holds for all functions $F_i$ that are real polynomials of degree at most $2$, which implies that the expectations and the correlation matrices of the two collections $(V_i)_{i\in I}$ and $(D^\ell(\gamma)V_i)_{i\in I}$ coincide.
\end{proof}
\begin{defi}
For all $\ell\in \frac12\N$, we define the matrix $\e(\ell)\in\C^{(2\ell+1)\times (2\ell+1)}$ as follows
\bega
\e(\ell)_{m,m'}=\delta_{-m,m'}(-1)^{\ell-m}, \quad \text{that is} \quad 
\e(\ell)=\begin{pmatrix}
\vdots &\vdots &\vdots &\vdots 
\\
0&0&0&-1
\\
0&0&1&0&\dots 
\\
0&-1&0&0 &\dots 
\\
1 &0 &0&0 &\dots 
\end{pmatrix},
\eega
where the coordinates are indexed by $m=-\ell,-\ell+1,\dots,\ell$, so that $\ell-m$ takes all integer values between $0$ and $2\ell+1$.
\end{defi}
\begin{thm}\label{thm:allDinvcor}
Let $\mathcal{V}=(V_i)_{i\in I}$ be a $2$-weakly invariant collection of random vectors $V_i\in\C^{2\ell_i+1}$. Then $\E\{V_i\}=0$, whenever $\ell_i\neq 0$ and  the correlation structure satisfies the following identities, for all $i,j\in I$.
\bega\label{eq:uno}
\E\left\{V_i(\overline{V_j})^T\right\}
&=
\frac{\delta_{\ell_i,\ell_j}}{2\ell_i+1}\mathbb{1}_{(2\ell_i +1)}
\left(\E\{\overline{\langle V_i,V_j\rangle}\}\right);
\\
\E\left\{V_i(V_j)^T\right\}
&=
\frac{\delta_{\ell_i,\ell_j}}{2\ell_i+1} \e(\ell_i)
\left(\sum_{k=-\ell_i}^{\ell_i}(-1)^{\ell_i+k}\E\{\langle e_k, V_i\rangle \langle e_{-k},V_j\rangle\}\right).
\eega 
In particular, $V_i$ and $V_j$ are uncorrelated if $\ell\neq\ell'$. Moreover, the $m^{th}$ component of $V_i$ is correlated only with the $m^{th}$ and the $(-m)^{th}$ components of $V_j$ and the correlation depends only on the parity of $\ell-m$. Finally, the variance of each components of $V_i$ depends only on $\ell_i$.
\end{thm}
\begin{proof}
Let us denote $V_i=V$ and  $V_j=V'$. By Theorem \ref{thm:VDgamma} we have
\bega 
\E\left\{V_i(\overline{V_j})^T\right\}=\E\left\{D^\ell(\gamma)V_i(\overline{D^\ell(\gamma)V_j})^T\right\}=\E\left\{\gamma_V(\overline{\gamma_{V'}})^T\right\}=\dots
\eega
The last expectation is with respect to the independent pair of random variables $\gamma$ and $(V,V')$. By taking first the one in $\gamma$ (i.e. using Fubini's theorem), Proposition \ref{prop:Schurorth} yields
\bega
\dots=\delta_{\ell,\ell'}\E\left\{K^\ell(V,V'))^T\right\}=\delta_{\ell,\ell'}\E\left\{\frac{\overline{\langle V,V'\rangle}}{2\ell+1}\mathbb{1}_{(2\ell+1)}\right\}.
\eega
This proves the first of the identities \eqref{eq:uno}. To prove the second identity, we use again Theorem \ref{thm:VDgamma}:
\bega 
\E\left\{\langle e_m, V_i\rangle \langle e_{-m'},V_j\rangle\right\}
=
\E\{(\rowd{m}(\gamma))^TV(D^{\ell'}_{m',\bullet}(\gamma))^TV'\}
=
\E\{V^T\rowd{m}(\gamma)(D^{\ell'}_{m',\bullet}(\gamma))^TV'\}
=\dots
\eega
Now, we use the same trick as before and apply the second formula of Theorem \ref{thm:Dcorr}:
\bega 
\dots =\E\left\{V^T\left(\rowd{m}(\gamma)(D^{\ell'}_{m',\bullet}(\gamma))^T\right)V'\right\}
=
\sum_{k,k'}\E\left\{\langle e_k, V\rangle\left(D^\ell_{m,k}(\gamma)D^{\ell'}_{m',k'}(\gamma)\right)\langle e_{k'}, V'\rangle\right\}
=\dots
\eega
\bega
\dots
=
\sum_{k,k'}\E\left\{\langle e_k, V\rangle
\left(\frac{\delta_{\ell,\ell'}\delta_{-m,m'} \delta_{-k,k'} (-1)^{2\ell -m+k}}{2\ell+1}\right)
\langle e_{k'}, V'\rangle\right\}
=\dots
\eega
\bega
\dots=
\sum_{k}\E\left\{\langle e_k, V\rangle
\left(\frac{\delta_{\ell,\ell'}\delta_{-m,m'} (-1)^{\ell -m}(-1)^{\ell+k}}{2\ell+1}\right)
\langle e_{-k}, V'\rangle\right\}.
\eega
\end{proof}
\section{Random spin weighted functions}\label{sec:randspf}
A  random function between two topological space $M, N$ is a measurable function
\be
\Omega\times M\to N,
\ee
where $(\Omega,\mathscr{S},\P)$ is a probability space and $M,N$ are endowed with their Borel $\sigma$-algebras. If $X$ is a random function, it is always possible to identify $\Omega$ with the set $N^M$ of all functions $\w\colon M\to N$, so that $X(\w,g)=\w(g)$, and $\mathcal{S}$ with a $\sigma$-algebra on $N^M$ having the property that the evaluation map $N^M\times M\ni (\w,g)\mapsto \w(g)\in N$ is measurable $\mathcal{S}\otimes \mathcal{B}(M)\to \mathcal{B}(N)$. This enables us to use a shortened notation $X\colon M\to N$. 
Given a subset $A\subset N^M$, we say that $X\in A$ \emph{almost surely}, if there exists a measurable set $S\subset N^M$ such that $A\subset S$ and $\P\{X\in S\}=1$.

A \emph{random spin weighted function} with \emph{spin weight $s$}, here called also \emph{random spin $s$ function} for short, $X_s$ is a random section of the $\spi{s}$ bundle, i.e. a random function
\be 
X_s\colon S^2\to \spi{s},
\ee
that is almost surely a section of the bundle $\spi{s}\to S^2$. By Theorem \ref{thm:pullback}, $X_s$ can be equivalently defined as a (complex) random field, i.e. a random function with values in $\C$
\be 
X_s\colon \sud\to \C,
\ee
such that $X_s\in\mathcal{R}(-s)$ almost surely.

We will focus on random fields that are isotropic. This word can be misleading in the case of $\sud$, because it can have two different meanings in the common language:
\begin{enumerate}[-]
\item A random field $X\colon G\to \C$ on a group $G$ is said to be isotropic when it is invariant in law under the left pull-back action:
\be\label{eq:leftinv}
 X(g^{-1} (\cdot))\law X(\cdot), \quad \forall g\in G.
\ee
\item A random function on the sphere $X\colon S^3\to \C$ is said to be isotropic if it is invariant in law under the pull-back action by elements of the group $SO(4)$ of (orientation preserving) isometries of the sphere: $X(\phi(\cdot) )\sim X(\cdot)$. In the case of $\sud\iso 2S^3$, this is equivalent to invariance under both the left and the right pull-back actions:
\be\label{eq:isotropy}
X(g^{-1} (\cdot) g')\law X(\cdot), \quad \forall g,g'\in\sud.
\ee
\end{enumerate}
We see that since $\sud$ is both a group and a sphere, the word \emph{isotropic} can be misleading, therefore we will not use it.
\begin{defi}
Let $X\colon \sud\to\C$ be any random function. We say that $X$ is \emph{left} \emph{invariant} if it satisfies condition \eqref{eq:leftinv} and \emph{right-invariant} if it satisfies the analogous condition for the right pull-back action. We say that $X$ is \emph{bi-invariant} if it satisfies condition \eqref{eq:isotropy}, i.e. if it is both left and right-invariant.
\end{defi}
The analogous notions of $n$-weak invariance and of invariance for a collection of random fields are considered here as stated in \cite[Definition 5.2]{libro}. In particular, we will be interested in the weakest among those notions of invariance, which takes into account just the correlation of pairs of variables. We recall it here, for the reader's convenience.
\begin{defi}
A collection of random fields $X^\ell\colon \sud \to \C$, for $\ell\in L$ is said to be $2$-weakly (left, right or bi)-invariant if $\E|X^\ell(g)| <\infty$ for every $\ell\in L$ and $g\in\sud$ and if the fields $X^\ell$ and $X^\ell\circ \phi$ have the same joint moments of order up to $n$:
\be\label{eq:2weak}
\E\left\{c^{i}\left(X^{\ell_1}\circ \phi(g_1)\right)c^j\left(X^{\ell_2}\circ\phi(g_2)\right)\right\}=\E\left\{c^i\left(X^{\ell_1}(g_1)\right)c^j\left(X^{\ell_2}(g_2))\right)\right\},
\ee
for every $\ell_1,\ell_2 \in L$,  every $i,j\in \{0,1\}$, every $g_1,g_2\in \sud$ and
for every $\phi\colon \sud\to \sud$ isometry of type $L_g$, $R_g$ or $L_g\circ R_{g'}$. Here $c\colon \C\to \C$ is the complex conjugation $c(z)=\overline{z}$.
\end{defi}
By the Stochastic Peter-Weyl theorem \cite[Theorem 5.5]{libro} and \cite[Proposition 5.4]{libro}, any $2$-weakly left-invariant field is automatically in $L^2$ almost surely: 
\be 
\exists \Omega_0 \subset \Omega \text{ s.t. } \P\{\Omega_0\}=1 \text{ and } X(\w,\cdot)\in L^2 \text{ for every $\w\in\Omega_0$}.
\ee 
In particular, it can be written in terms of the Hilbert basis \eqref{eq:hypersph} for some collection of random variables $a^\ell_{m,s}\in \C$:
\be\label{eq:series}
X(g)=\sum_{\ell,m,s} a^\ell_{m,s} \phi^\ell_{m,s}(g).
\ee
This consideration, together with Theorem \ref{thm:Disharmonic} and Theorem \ref{prop:dirrep} prove Theorem \ref{thm:main1}.
The series \eqref{eq:series} converges almost surely in $L^2(\sud)$ and almost surely pointwise in $\C$, for every $g\in \sud$. Let 
\be 
X^\ell:=\sum_{m,s}a^\ell_{m,s}\phi^\ell_{m,s}\colon \sud\to \C^{(2\ell+1)\times (2\ell+1)}
\ee
be the projections of $X$ onto the subspaces $\matrsp$ of matrix coefficients. It follows that if the field $X$ is ($2$-weakly) left, right or bi invariant, then the collection of fields $X^\ell$, for $s\in\frac12\Z$ is ($2$-weakly) left, right or bi invariant as well and viceversa, see \cite[Proposition 5.4]{libro}. Remarkably, again by \cite[Proposition 5.4]{libro}, the fields $X^\ell$ are uncorrelated for different $\ell\in\frac12\N$, therefore Theorem \ref{thm:maindell}. A further decomposition gives
\be 
\xold{s}:=\sum_{m}a^\ell_{m,s} \phi^\ell_{m,s}; \qquad \xowd{m}:=\sum_{s} a^\ell_{m,s} \phi^\ell_{m,s}.
\ee
The fields $\xold{s}, \xowd{m}\colon \sud\to \C^{2\ell+1}$ are, respectively, the projections of $X$ onto the subspaces $\coldsp{s}$ and $\rowdsp{m}$. It follows that if the field $X$ is ($2$-weakly) left-invariant, then the collection of fields $\xold{s}$, for $s\in\frac12\Z$ is ($2$-weakly) left-invariant as well. However, the fields $\xold{s}, \xold{s'}$ need not to be uncorrelated. 
\begin{example}
If $X=\sum_{m}a^\ell_{m,s} \phi^\ell_{m,s}$ is left-invariant, then $X'=\sum_{m}a^\ell_{m,s} (\phi^\ell_{m,s}+\phi^\ell_{m,s'})$ is again left-invariant. In the second case, the corresponding fields $\xold{s}$ and $\xold{s'}$ are clearly correlated.
\end{example}
The same discussion can be repeated for right-invariance.
Notice that the law of $X$ can be thought as the law of the collections of random vectors
\be 
\aold{s}=\begin{pmatrix}
a^\ell_{-\ell,s} \\ \vdots \\ a^\ell_{\ell,s}
\end{pmatrix}\in \C^{2\ell+1} \quad \text{or}\quad \aowd{m}=\begin{pmatrix}
a^\ell_{m,-\ell} \\ \vdots \\ a^\ell_{m,\ell}
\end{pmatrix}\in \C^{2\ell+1},
\ee
with $\ell\in\frac12\N$, $m,s=-\ell,\dots, \ell$, corresponding respectively to the laws of $\xold{s}$ and $\xowd{s}$.
This observation allows us to adopt the point of view of the Section \ref{sec:dinva} above. 
The following theorem is similar to \cite[Lemma 6.3]{libro}.
\begin{lemma}\label{lem:ainva}
Let $\gamma\in\sud$ be a random element distributed with the Haar measure. A square integrable random field $X$ is, respectively, $2$-weakly or strongly left-invariant, if and only if the collection of random vectors $\overline{\aold{s}}$, with $\ell\in\frac12$ and $s\in\{-\ell,\dots,\ell\}$ is $2$-weakly or strongly $D$-invariant:
\be\label{eq:dga} 
\left(\overline{\aold{s}}\right)_{\ell,s}\twik \text{ or } \law\left(D^\ell(\gamma)\overline{\aold{s}}\right)_{\ell,s}.
\ee
 Similarly, if $X$ is $2$-weakly or strongly right-invariant, then the collection of random vectors $\left(\aowd{m}\right)_{\ell,m}$ is $2$-weakly or strongly $D$-invariant:
\be\label{eq:dgaright} 
\left(\aowd{m}\right)_{\ell,m}\twik \text{ or } \law \left(D^\ell(\gamma)\aowd{m}\right)_{\ell,m}.
\ee
\end{lemma}
\begin{proof}
We have for any $g,z\in\sud$:
\bega 
\xold{s}(g^{-1}z)=(\aold{s})^T\fold{s}(g^{-1}z)=(\aold{s})^TD^\ell(g^{-1})\fold{s}(z)=\left(\overline{D^\ell(g)}\aold{s}\right)^T\fold{s}.
\eega
Similarly, from the point of view of the right pull back action:
\bega
\xowd{m}(zg)=(\fowd{m}(z))^T\aowd{m}(zg)=(\fowd{m}(z))^TD^\ell(g)\aowd{m}(z),
\eega
because $\fowd{m}$ is the $m^{th}$ row of the matrix $\phi^\ell$. We conclude the proof by an application of Theorem \ref{thm:VDgamma}.
\end{proof}
We immediately get the following Corollary. Moreover, this proves the first part of Theorem \ref{thm:mainleft}.
\begin{cor}
Let $\gamma\in\sud$ be a random element distributed with the Haar measure.
A square integrable random field $X$ is, respectively, $2$-weakly or strongly left-invariant, if and only if 
\be 
X(\cdot)
\twik \text{ or } \law X(\gamma(\cdot) ).
\ee
It is, respectively, $2$-weakly or strongly right-invariant, if and only if 
\be 
X(\cdot)
\twik \text{ or } \law X((\cdot)\gamma ).
\ee
\end{cor}
\begin{thm}\label{thm:leacor}
Let $X\colon \sud \to \C$ be a $2$-weakly left-invariant random field. Then $\E\{a^\ell_{m,s}\}=0$ whenever $\ell\neq 0$. Moreover, the random variables $a^\ell_{m,s}$ have the following correlation structure:
\be\label{eq:leacor}
\E\left\{\overline{\left(a^\ell_{m,s}\right)} a^{\ell'}_{m',s'}\right\}=
\delta_{m,m'}
\frac{\delta_{\ell,\ell'}}{2\ell+1}\E\left\{
\langle \xold{s},\xold{s'}\rangle_{L^2(\sud)}
\right\}.
\ee
In particular, the variance of $a^\ell_{m,s}$ does not depend on $m$. Furthermore:
\be\label{eq:circsym}
\E\left\{a^\ell_{m,s}(a^{\ell'}_{m',s'})\right\}=
\delta_{-m,m'}(-1)^{\ell-m}\frac{\delta_{\ell,\ell'}}{2\ell+1}
\left(\sum_{k=-\ell}^{\ell}(-1)^{\ell+k}\E\left\{a^\ell_{k,s}( a^\ell_{-k,s'})\right\}\right).
\ee
\end{thm}
\begin{proof}
A combination of Lemma \ref{lem:ainva} and Theorem \ref{thm:allDinvcor} gives the thesis. 
\end{proof}
By repeating the same arguments, we obtain an analogous statement for right-invariant random fields.
\begin{thm}\label{thm:riacor}
Let $X\colon \sud \to \C$ be a $2$-weakly right-invariant random field. Then $\E\{a^\ell_{m,s}\}=0$ whenever $\ell\neq 0$. Moreover, the random variables $a^\ell_{m,s}$ have the following correlation structure:
\be\label{eq:riacor}
\E\left\{\overline{\left(a^\ell_{m,s}\right)}a^{\ell'}_{m',s'}\right\}=
\delta_{s,s'}
\frac{\delta_{\ell,\ell'}}{2\ell+1}\E\left\{
\langle \xowd{m},\xowd{m'}\rangle_{L^2(\sud)}
\right\}.
\ee
In particular, the variance of $a^\ell_{m,s}$ does not depend on $s$. Furthermore:
\be\label{eq:ricircsym}
\E\left\{a^\ell_{m,s}(a^{\ell'}_{m',s'})\right\}=
\delta_{-s,s'}(-1)^{\ell-s}\frac{\delta_{\ell,\ell'}}{2\ell+1}
\left(\sum_{k=-\ell}^{\ell}(-1)^{\ell+k}\E\left\{a^\ell_{m,k}( a^\ell_{m',-k})\right\}\right).
\ee
\end{thm}
\begin{thm}\label{thm:biacor}
Let $X\colon \sud \to \C$ be a $2$-weakly bi-invariant random field. Then $\E\{a^\ell_{m,s}\}=0$ whenever $\ell\neq 0$. Moreover, the random variables $a^\ell_{m,s}$ have the following correlation structure:
\be\label{eq:biacor}
\E\left\{\overline{\left(a^\ell_{m,s}\right)} a^{\ell'}_{m',s'}\right\}=
\delta_{m,m'}\frac{\delta_{\ell,\ell'}}{(2\ell+1)^2}\delta_{s,s'}\E\left\{
\| X^\ell\|^2_{L^2(\sud)}
\right\}.
\ee
In particular, the variance of $a^\ell_{m,s}$ depends only $\ell$. Furthermore:
\be\label{eq:bicircsym}
\E\left\{a^\ell_{m,s}(a^{\ell'}_{m',s'})\right\}=
\delta_{-m,m'}\frac{\delta_{\ell,\ell'}}{(2\ell+1)^2}
\delta_{-s,s'}(-1)^{\ell-m}(-1)^{\ell-s}
\E\left\{\langle \overline{X^\ell}, X^\ell \rangle_{L^2(\sud)}\right\}i^{2\ell}.
\ee
\end{thm}
\begin{proof}
Equation \eqref{eq:biacor} follows by the fact that $X$ is both left and right invariant, thus both Theorem \ref{thm:leacor} and Theorem \ref{thm:riacor} hold.
To prove the second equation, let us observe that, because of \eqref{eq:circsym} and \eqref{eq:ricircsym}, we have that there exists a constant $C\in\C$ such that:
\be 
\E\left\{a^{\ell}_{m,s}
a^{\ell}_{-m,-s}\right\}=C (-1)^{\ell-m}(-1)^{\ell-s}.
\ee
We find the value of $C$ with the following computation.
\bega
\E\left\{\left\langle \overline{X^\ell},X^\ell\right\rangle\right\}
=\E\left\{\left\langle \sum_{m,s}\overline{a^{\ell}_{m,s}\phi^\ell_{m,s}},\sum_{m',s'}\overline{a^{\ell}_{m',s'}\phi^\ell_{m',s'}}\right\rangle\right\}
=\dots
\eega
Using Proposition \ref{prop:otherD} this is
\bega
\dots =
\sum_{m,m',s,s'}
\E\left\{\left\langle \overline{a^{\ell}_{m,s}}
\phi^\ell_{-m,-s}(-1)^{2\ell+m-s},a^{\ell}_{m',s'}\phi^\ell_{m',s'}\right\rangle\right\}
=
\sum_{m,s}(-1)^{2\ell+m-s}
\E\left\{a^{\ell}_{m,s}
a^{\ell}_{-m,-s}\right\}
=
\dots
\eega
\bega
\dots
=
\sum_{m,s}(-1)^{2\ell+m-s}
C(-1)^{\ell-m}(-1)^{\ell-s}=C(2\ell+1)^2i^{2\ell}.
\eega
\end{proof}
\subsection{Spectral probability}\label{sec:spectralprob}
We use the Hilbert basis formed by the normalized Wigner functions, i.e. the spherical harmonics $\phi^\ell_{m,s}$, to define a notion of \emph{spin} for every $f\in L^2(\sud)$. 
\begin{defi}
Let $F\in L^2(\sud)$, and let $\|F\|$ denote its $L^2(\sud)$ norm. There exist coefficients $a_{m,s}^\ell\in \C$ such that
\be 
F=\sum_{\ell,m,s}a_{m,s}^\ell \phi^\ell_{s,m}.
\ee 
We define the \emph{spectral probability} of $F$ to be the probability measure $\TS{F}$ on the space $\hat{S}^3=\{(\ell,m,s)\}$ (defined in \eqref{eq:GdualG}) such that
\be
\TS{F}(\{(\ell,m,s)\}):=\frac{|a_{m,s}^\ell|^2}{\|F\|^2}.
 \ee
Similarly, we define the \emph{left spin}  and the \emph{right spin} of $F$ as the probability measures on $\frac12\Z$ such that for every singleton $m,s\in \frac{1}{2}\Z$, we have
\be\label{eq:LRspinmeasure}
\LS{F}(\{m\}):=\frac{\sum_{\ell,s}|a_{m,s}^\ell|^2}{\|F\|^2} \quad \text{and} 
\quad 
\RS{F}(\{s\}):=\frac{\sum_{\ell,m}|a_{m,s}^\ell|^2}{\|F\|^2}.
\ee
Moreover, we call \emph{bi spin} of $F$ , the probability on $\frac12\Z\times\frac12\Z$ defined for every singleton $(m,s)$ as
\be 
\BS{F}(\{(m,s)\}):=\sum_{\ell\in\frac12\N}\frac{|a_{m,s}^\ell|^2}{\|F\|^2}.
\ee
\end{defi}
In particular, a function $F$ has pure right spin $=-s$, and thus it is section of $\spi{s}$, if and only if $\RS{F}$ is the delta measure on $s$ (and similarly in the case of pure left spin).


Let us consider a square integrable random field, i.e. a random function $X\colon \sud\to \C$, such that $X\in L^2(\sud)$ almost surely, so that
\be 
X=\sum_{\ell,m,s}a^{\ell}_{m,s}\phi^\ell_{m,s}.
\ee
Then we define another associated spectral probability on $\hat{S}^3$:
\be
\TSE{X}(\{\ell,m,s\})=\frac{\E|a_{m,s}^\ell|^2}{\E\|X\|^2}.
\ee
This has to be compared with the expectation of the random probability $\TS{X}$, that is
\be
\E\TS{X}(\{\ell,m,s\})=\E\left(\frac{|a_{m,s}^\ell|^2}{\|X\|^2}\right). 
\ee
Similarly, we define $\LSE{X}$, $\RSE{X}$, and $\BSE{X}$.
Thus, we have $2$ probability measures on $\hat{S}^3$, that are associated to the random field $X$ and that give a sense of the distribution of the left and right spin of $X$ and of its homogeneous components, i.e of the relative magnitude of the random variables $a^\ell_{m,s}$

In general, the probabilities $\TSE{X}$ and $\E \TS{X}$ might be different, in that the first takes into account only the correlation structure of the variables $a^{\ell}_{m,s}$, i.e. it depends on the field $X$ up to $2$-weak equivalence. In fact, even less, it just depends on the marginal distributions of the coefficients. For this reason, we will call $\TSE{X}$, the \emph{weak  spectral probability} of $X$. On the other hand, the expected spectral probability $\E \TS{X}$ depends on higher moments and on the joint distribution, thus it should be considered as a more descriptive data, thus we call it \emph{strong spectral probability} of $X$.
\begin{thm}\label{thm:weakuniform}
Let $X$ be $2$-weakly left-invariant random field. Then, the weak spectral measure $\TSE{X}$ is uniform on the sets
\be\label{eq:setform}
\{\ell\}\times\{-\ell,\dots,\ell\}\times \{s\},
\ee
for all $\ell\in\frac12\N$.
If $X$ is strongly left-invariant, then the same is true for the strong spectral measure $\E\TS{X}$. The analogous statement is true when $X$ is right-invariant, in which case the measure is uniform on the sets of the form
\be 
\{\ell\}\times \{m\}\times \{-\ell,\dots ,\ell\}.
\ee
\end{thm}
\begin{proof}
From Theorem \ref{thm:leacor} we see that the variance of $a^\ell_{m,s}$ does not depend on $m$, when $X$ is $2$-weakly left-invariant, which means exactly that $\TSE{X}(\ell,m,s)$ is uniform on sets of the form \eqref{eq:setform}
In case $X$ is strongly left-invariant, we have that the field 
\be 
Y=\frac{1}{\|X\|}X
\ee
 is again strongly left-invariant. Therefore we can apply the first part of the theorem to it, but in this case
 \be 
 \LSE{Y}=\E\LS{Y}=\E\LS{X}.
 \ee
\end{proof}
\begin{cor}
Let $X=\sum_{\ell,s} \xold{s}$ be $2$-weakly left-invariant random field. Then for all $\ell\in\frac12\N$ and $s\in\{-\ell,\dots \ell\}$, the measure $\LSE{\xold{s}}$ is the uniform probability on $\{-\ell,\dots,\ell\}$.
If $X$ is strongly left-invariant, then also $\E\LS{\xold{s}}=\LSE{\xold{s}}$ is uniform. The analogous statement is true when $X$ right-invariant.
\end{cor}
\begin{thm}
Any probability measure $\mu$ on $\{-\ell,\dots,\ell\}$ can be realized as $\mu=\E\RS{X^\ell}$, with $X^\ell$ being a strongly left-invariant random field.
\end{thm}
\begin{proof}
Let $\gamma_{-\ell},\dots,\gamma_\ell\in\sud$ be independent uniform random elements of $\sud$. Define \be 
X(g):=\sum_{s}\mu(\{s\})^{\frac12}\phi^\ell_{s,s}(\gamma\cdot g)=\sum_{s}\mu(\{s\})^{\frac12}(\rowd{s}(\gamma))^T\fold{s}( g).
\ee
$X$ is strongly left-invariant by construction. Moreover, the coefficients are $a^\ell_{m,s}=\mu(\{s\})^{\frac12}D^\ell_{s,m}(\gamma)$, so that the random vectors $\aold{s}$ are orthogonal almost surely and have constant length, due to the unitarity of $D^\ell(g)$, for any $g\in\sud$. It follows that $\|X\|^2_{L^2(\sud)}=C$ is constant almost surely. Therefore, the strong and weak right spin measures are equal to $\mu$: 
\be 
\E\RS{X}(\{s\})=\sum_{m}\E\left(\mu(\{s\})|D^\ell_{s,m}(\gamma)|^2\right)=\mu(\{s\}).
\ee
\end{proof}
This concludes the proof of Theorem \ref{thm:mainleft}. The following result implies Theorem \ref{thm:main2}.
\begin{thm}\label{thm:stronguniform}
Let $X$ be a $2$-weakly bi-invariant random field. Then the weak spectral probability $\TSE{X}$ is uniform on all sets of the form
\be 
\{\ell\}\times \{-\ell,\dots ,\ell\}\times \{-\ell,\dots ,\ell\}.
\ee
If $X$ is strongly left-invariant, then the same is true for the strong spectral probability $\E\TS{X}$.
\end{thm}
\begin{proof}
Follows from Theorem \ref{thm:biacor} and the same argument used in the proof of the previous theorem.
\end{proof}
\begin{cor}\label{cor:stronguniform}
Let $X=\sum_{\ell} X^\ell$ be a $2$-weakly bi-invariant random field. Then for all $\ell\in\frac12\N$ and $m,s\in\{-\ell,\dots \ell\}^2$, the measure $\BSE{X^\ell}$ is the uniform probability on $\{-\ell,\dots,\ell\}$.
If $X$ is strongly left-invariant, then $\E\BS{\xold{s}}=\BSE{\xold{s}}$.
\end{cor}
\subsection{The Gaussian case}
Let us turn our focus to everyone's favorite random fields, the Gaussian ones. We say that $V\in\R^N$ is a real Gaussian random vector if all linear combinations of its components are Gaussian random variables. For simplicity, we will only consider the case of centered Gaussian.
We say that a complex random variable $a\in\C$ is complex Gaussian if and only if, it is a circularly symmetric Gaussian random vector in $\R^2$, i.e. there exists $\sigma\in\R$ and $\xi_1,\xi_2\sim N(0,\sigma^2)$ independent such that
\be 
a=\sigma\left(\frac{1}{\sqrt{2}}\xi_1+i\frac{1}{\sqrt{2}}\xi_2\right),
\ee 
in this case we write $a\sim N_\C(0,\sigma)$. We say that $V\in\C^N$ is a complex Gaussian random vector if all $\C$-linear combinations of its components are complex Gaussian. 
The distribution of a real Gaussian random vector $V\in \C^N=\R^{2N}$ is determined by the correlation matrices
\be 
K:=\E\{V\overline{V}^T\} \quad\text{ and }\quad C:=\E\{VV^T\}.
\ee
A real Gaussian random vector is complex Gaussian if and only if $C=0$
and we write $V\sim N_\C(0,K)$. 
\begin{defi}
A random field $X\colon \sud\to\C$ is Gaussian (complex or real) if for every finite set of points $g_i\in\sud$, the random vector is Gaussian (complex or real)
\be 
(X(g_1),\dots,X(g_N))\in \C^N.
\ee
\end{defi}
It is straightforward to see that an almost surely square integrable random field $X\colon \sud \to \C$ is Gaussian (complex or real) if and only if the coefficients $a^\ell_{m,s}$ of the decomposition
\be 
X=\sum_{\ell,m,s}a^\ell_{m,s}\phi^\ell_{m,s}
\ee
 are a family of jointly Gaussian (complex or real) random variables\footnote{Clearly if $a^\ell_{m,s}$ are jointly Gaussian, then $X$ is a Gaussian random field. The converse, follows from the fact that 
 \be 
 a^\ell_{m,s}=\int_{\sud}X(g)\overline{\phi^\ell_{m,s}}(g)d\mu(g),
 \ee
 hence $a^\ell_{m,s}$ can be expressed the almost sure limit of a sequence of linear combination of random variables of the form $X(g)$, therefore it is Gaussian.}. 
In particular, if $X$ is complex Gaussian, then 
 \be 
\E\{a^\ell_{m,s}a^{\ell'}_{m',s'}\} =0,
 \ee
 for every $m,m',s,s',\ell,\ell'\in\frac12\Z$. 
\begin{thm}\label{thm:gazzo}
A complex Gaussian random field $X\colon \sud \to \C$ is $2$-weakly (left, right or bi)-invariant if and only if it is strongly (left,right or bi)-invariant. In particular, it is left-invariant if and only if the fields $X^\ell$ are independent with
\be\label{eq:gavr} 
\aold{s}\sim N\left(0,\sigma(\ell,s)^2\mathbb{1}_{(2\ell+1)}\right),
\ee
for some $\sigma(\ell,s)\in \R$
and there are constants $K(s,s')\in\C$ such that
\be \label{eq:basta}
\E\{a^\ell_{\bullet,s}(\overline{a^\ell_{\bullet,s'})^T}\}=K(s,s')\mathbb{1}_{(2\ell+1)}.
\ee
 The analogous statement holds if $X$ is right-invariant. Moreover, $X$ is bi-invariant if and only if all the variables $a^\ell_{m,s}$ form an independent family and equation \eqref{eq:gavr} holds with $\sigma(\ell,s)=\sigma(\ell)$.
\end{thm}
\begin{proof}
The first statement follows from Lemma \ref{lem:ainva} and the fact that the distribution of a Gaussian vector is uniquely determined by the correlation structure of the variables $a^\ell_{m,s}$. By a further examination of the characterization given in Theorems \ref{thm:leacor} and Theorem \ref{thm:biacor} we complete the proof. 
\end{proof}
In \cite{BalTra,bamavara} it is proved that moreover the only left-invariant random fields for which the variables $a^\ell_{m,s}$ are independent are the Gaussian ones. This, combined with Theorem \ref{thm:gazzo} and Theorem \ref{thm:main1}, proves Corollary \ref{cor:main3}.
\bibliographystyle{plain}
\bibliography{IsotropyVSIsotropy}

\end{document}